\documentclass[12pt,a4paper, oneside]{amsart}

\usepackage{amsmath, nicefrac, amsthm, verbatim, amsfonts, mathtools, amssymb, upgreek, xcolor, bbm}
\usepackage{graphics, xspace, enumerate}
\usepackage{stix}
\usepackage{dsfont}
\usepackage[a4paper,margin=2.5cm]{geometry}
\usepackage[bb=boondox]{mathalfa}

\usepackage{graphicx}
\usepackage[colorlinks=true,citecolor=red,urlcolor=blue,linkcolor=red,bookmarksopen=true,unicode=true,pdffitwindow=true]{hyperref}
\usepackage[english]{babel}
\usepackage[languagenames,fixlanguage]{babelbib}
\hypersetup{pdfauthor={}}
\hypersetup{pdftitle={}}

\usepackage{seqsplit,cleveref}

\theoremstyle{plain}
\newtheorem{theorem}{Theorem}[section]

\newtheorem{lemma}[theorem]{Lemma}
\newtheorem{proposition}[theorem]{Proposition}
\theoremstyle{definition}
\newtheorem{remark}[theorem]{Remark}

\newtheorem{definition}[theorem]{Definition}

\newcommand {\Prob} {\ensuremath{\mathbb{P}}}
\newcommand {\R} {\ensuremath{\mathbb{R}}}
\newcommand {\ZZ} {\ensuremath{\mathbb{Z}}}
\newcommand {\N} {\ensuremath{\mathbb{N}}}

\newcommand{\process}[1]{\{#1(x,t)\}_{t\geq0}}

\newcommand{\ttup}[1]{\textup{(}#1\textup{)}}
\newcommand{\df}{\coloneqq}

\newcommand{\n}{\mathrm{n}}
\newcommand{\m}{\mathrm{m}}
\newcommand{\dd}{\mathrm{d}}
\newcommand{\E}{\mathrm{e}}
\newcommand{\Id}{\mathbb{I}}

\newcommand{\PP}{\mathcal{P}}

\newcommand{\X}{\mathrm{X}}
\newcommand{\A}{\mathcal{A}}

\newcommand{\D}{\mathrm{d}}
\newcommand{\DD}{\mathrm{D}}
\newcommand{\HS}[1]{\lVert #1 \rVert_{\mathrm{HS}}}

\numberwithin{equation}{section}

\title[Periodic homogenization  of linear degenerate PDE\MakeLowercase{s}]{A CLT for degenerate diffusions  with periodic coefficients, and application to  homogenization of  linear  PDE\MakeLowercase{s}}

\author[N.\ Sandri\'{c}]{Nikola Sandri\'{c}}
\address[Nikola\ Sandri\'{c}]{Department of Mathematics\\University of Zagreb\\ Zagreb\\Croatia}
\email{nsandric@math.hr}

\author[I.\ Valenti\'c]{Ivana\ Valenti\'c}
\address[Ivana\  Valenti\'c]{
	Department of Mathematics\\University of Zagreb\\ Zagreb\\Croatia}
\email{ivana.valentic@math.hr}


\begin{document}
\allowdisplaybreaks[4]

\begin{abstract}
	In this article, we obtain a functional CLT  for a class of degenerate diffusion processes with periodic coefficients, thus generalizing the already classical  results in the context of uniformly elliptic diffusions. As an application, we also  discuss periodic homogenization of a class of linear degenerate elliptic and parabolic PDEs.  
\end{abstract}

\maketitle

\section{Introduction}\label{S1} 

	Let $\mathcal{L}^{\varepsilon}$, $\varepsilon>0$, be a second-order elliptic differential operator  of the form
\begin{equation}\label{ES1.1}
\mathcal{L}^{\varepsilon}\,=\,
2^{-1}\mathrm{Tr}\bigl(a(\cdot/\varepsilon)\,\nabla\nabla^\mathrm{T}\bigr)+
\bigl(\varepsilon^{-1}b(\cdot/\varepsilon)+c(\cdot/\varepsilon)\bigr)^{\mathrm{T}}\nabla\,.
\end{equation} 
The main goal of this article is to discuss periodic homogenization (that is, asymptotic behavior of the solution as $\varepsilon\to0$) of the associated elliptic boundary-value problem \begin{equation}\label{ES1.2}
\begin{aligned}
\mathcal{L}^{\varepsilon} u^\varepsilon(x)+
e(x/\varepsilon)\,u^{\varepsilon}(x)+f(x)&\,=\,0\,,\qquad x\in\mathscr{D}\,,\\
u^\varepsilon(x)&\,=\, g(x)\,,\qquad x\in\partial\mathscr{D}\,,
\end{aligned}
\end{equation}
as well as the  parabolic initial-value problem
\begin{equation}\label{ES1.3} 
\begin{aligned}
\partial_tu^\varepsilon(x,t)&\,=\,\mathcal{L}^{\varepsilon} u^\varepsilon(x,t)+\bigl(\varepsilon^{-1}d(x/\varepsilon)+e(x/\varepsilon)\bigr)u^{\varepsilon}(x,t)+f(x)\\
u^\varepsilon(x,0)&\,=\, g(x)\,,\qquad x\in\R^\n\,,
\end{aligned}
\end{equation}
in the case of degenerate (possibly vanishing on a set of positive Lebesgue measure) diffusion coefficient $a(x)$. 
Our approach to this problem relies on probabilistic techniques: we first show that the (appropriately centered) diffusion process associated to  $\mathcal{L}^{\varepsilon}$ satisfies a functional CLT with Brownian limit as $\varepsilon\to0$ (see \Cref{T3.1,T3.5}), and then by employing probabilistic representation (the Feynman-Kac formula) of the (viscosity) solutions to the problems in \cref{ES1.2,ES1.3} obtained in  \cite[Chapter 3]{Pardoux-Rascanu-Book-2014}  we conclude the homogenization result (see \Cref{T4.4,T4.5}). This  idea goes back to M. I. Fre\u{\i}dlin  \cite{Freidlin-1964} (see also \cite[Chapter 3]{Bensoussan-Lions-Papanicolaou-Book-1978}). In the non-degenerate (uniformly elliptic) case these steps can be carried out by combining classical PDE results (existence of a smooth solution to the corresponding Poisson equation) and the fact that the underlying diffusion process  does not show a singular behavior in its motion, that is, it is irreducible (see \cite[Chapter 3]{Bensoussan-Lions-Papanicolaou-Book-1978} and \cite{Bhattacharya-1985} for a detailed exposition). In the case of a degenerate diffusion part, this deficiency is compensated by the assumption   
 that the underlying diffusion process  with positive probability reaches the   part of the state space  where the diffusion term is non-degenerate   (see assumption \ttup{A3}). In oder words, this condition ensures irreducibility of  the process (see \Cref{S2} for details). Also, in this case it is not clear that we can rely on PDE techniques therefore the analysis of  a  solution to the corresponding  Poisson equation is completely based on stochastic analysis tools, in particular Motoo's theorem \cite[Proposition 3.56]{Cinlar-Jacod-Protter-Sharpe-1980} (see \Cref{S3}).


\subsection{Literature review}  
Our work contributes to the classical theory of periodic homogenization. Most of the existing literature on this subject focuses on the problem of homogenization of   non-degenerate PDEs; for instance, see the classical monographs \cite{Allaire-2002-Book}, \cite{Bensoussan-Lions-Papanicolaou-Book-1978},   \cite{Jikov-Kozlov-Oleinik-1994-Book} and  \cite{Tartar-2009-Book}. However, in the  recent years  there have been  developments in understanding the
homogenization of degenerate PDEs.
We refer the readers to \cite{DeArcangelis-SerraCassano-1992}, \cite{Paronetto-1999}, \cite{Paronetto-2004} and \cite{Paronetto-SerraCassano-1998}  for a PDE approach to this problem, and 
 \cite{Chiarini-Deuschel-2016}, \cite{Delarue-Rhodes-2009}, \cite{Rhodes-2007} and \cite{Rhodes-2009} for a probabilistic approach.
 However, in all these works the major limitation is that  the diffusion term can fully degenerate (vanish) on a ``small'' part of the domain only.
 In the first five references it is allowed that it vanishes on a set of Lebesgue measure zero only and in the rest of the domain it must have a full rank. While  in \cite{Delarue-Rhodes-2009}, \cite{Rhodes-2007} and \cite{Rhodes-2009} it is allowed that it  degenerates  everywhere, but its rank must be greater than or equal to one except maybe on a set of Lebesgue measure zero. 
In the present work we partly fill this gap and focus on the case when the diffusion part  vanishes on a set of positive Lebesgue measure.
 In the closely related article \cite{Hairer-Pardoux-2008} (see also  \cite{Pardoux-Rhodes-Sow-2009} and \cite{Pardoux-Sow-2011}  in the context of semilinear elliptic and parabolic PDEs), by also employing probabilistic methods, the authors are concerned with  the same questions we discuss in this article. However, unfortunately, there seems to be a doubt about their proof of the functional CLT in \cite[Theorem 3.1]{Hairer-Pardoux-2008} (see \Cref{S3.2} for details).
In this article, under slightly weaker assumptions (and by employing different techniques)  we  resolve this issue, or at least suggest an alternative approach to the problem. We distinguish two cases: (i) $c(x)\equiv0$, and (ii) $c(x)\nequiv0$. In the first case, in \Cref{T3.1} we obtain the functional CLT under the assumptions in \ttup{A1}-\ttup{A3}. Here, \ttup{A1} should be compared to  \ttup{H.1} from  \cite{Hairer-Pardoux-2008}, and \ttup{A2}-\ttup{A3} to \ttup{H.2}. Note that the condition in \ttup{H.2} assumes existence of a fixed time $t_0>0$ such that for every $\varepsilon\in[0,\varepsilon_0]$, for some $\varepsilon_0>0$, the corresponding diffusion process  observed at $t_0$ is with  positive probability in a part of the state space where the diffusion term is non-degenerate, which is a slightly stronger assumption than \ttup{A3}. 
Observe also  that in addition to \ttup{H.1}-\ttup{H.2} the authors in \cite{Hairer-Pardoux-2008} assume  a regularizing  condition \ttup{H.3}, which we do not require in this case.
On the other hand, the case
$c(x)\nequiv0$   is technically more delicate. Namely, an analogous analysis as in \Cref{T3.1}  cannot  be performed in this situation.  To overcome this difficulty we adopt assumption   \ttup{H.3} 
from \cite{Hairer-Pardoux-2008} (see \ttup{A4}), which allows us to conclude an It\^{o}-type formula for the diffusion process associated to the operator $2^{-1}\mathrm{Tr}(a(\cdot)\,\nabla\nabla^\mathrm{T})+
(b(\cdot)+\varepsilon c(\cdot))^{\mathrm{T}}\nabla$ (see \Cref{L3.4}). 
With this in hand, and basing on the ideas from \cite{Freidlin-1964}, we are then able to obtain the required functional CLT (see \Cref{T3.5}) and  conclude the homogenization results in \Cref{T4.4,T4.5}.

The results of this article can also be found 
as a part of the second-named author's doctoral dissertation \cite{Valentic-PhD-2020}, where the problem of periodic homogenization of a class of L\'evy-type operators has been discussed.


\subsection{Notation} 
We summarize some  notation used throughout the article. 
We use $\R^\n$, $\n\in\N$, to denote real-valued
$n$-dimensional  vectors, and write $\R$ for $\n=1$. All  vectors will be column vectors. cis denoted by $\lvert\cdot\rvert$.
By $M^{\mathrm{T}}$ and $\HS{M}\df (\mathrm{Tr}\,MM^{\mathrm{T}})^{1/2}$ we denote the transpose and the Hilbert-Schmidt norm of  a $n\times m$-matrix $M$, respectively. For a 
square matrix $M$,
 $\mathrm{Tr}\, M$ stands for its trace. For a set $A\subseteq\R^\n$,  the symbols
 $A^{c}$, $\mathbb{1}_{A}$, $\overline{A}$ and $\partial A$  stand for
 the complement, indicator function,  (topological) closure and (topological) boundary of $A$, respectively.
We let  $\mathfrak{B}(\R^\n)$ and $\mathcal{B}(\R^\n,\R^\m)$ denote the Borel $\sigma$-algebra on $\R^\n$ and the space of $\mathfrak{B}(\R^\n)/\mathfrak{B}(\R^\m)$-measurable functions, respectively. Also, for $A\subseteq\R^\n$, $\mathfrak{B}(A)$ stands for $\{A\cap B\colon B\in\mathfrak{B}(\R^n)\}$.
For a Borel  measure $\upmu(\D x)$
on $\mathfrak{B}(\R^\n)$ and $f=(f_1,\dots,f_\m)^{\mathrm{T}}\in\mathcal{B}(\R^\n,\R^\m)$, we often use the convenient notation
$\upmu(f)=\int_{\R^\n} f(x)\,\upmu(\D{x})\df (\int_{\R^\n} f_1(x)\,\upmu(\D{x}),\dots,\int_{\R^\n} f_\m(x)\,\upmu(\D{x}))^\mathrm{T}$.
For $f\in\mathcal{B}(\R^\n,\R^\m)$ we let $\|f\|_\infty\df\sup_{x\in\R^\n}|f(x)|$ denote its supremum norm, and 
$\mathcal{B}_b(\R^\n,\R^\m)$ stands for $\{f\in\mathcal{B}(\R^\n,\R^\m)\colon \|f\|_\infty<\infty\}$. We use  $\mathcal{C}_b^k(\R^\n,\R^\m)$, $\mathcal{C}_{u,b}^k(\R^\n,\R^\m)$, $\mathcal{C}_{\infty}^k(\R^\n,\R^\m)$ and   $\mathcal{C}_{c}^k(\R^\n,\R^\m)$, $k\in\N_0\cup\{\infty\}$, to denote the subspaces of $\mathcal{B}_b(\R^\n,\R^\m)\cap \mathcal{C}^k(\R^\n,\R^\m)$ of all  $k$ times differentiable functions such that all derivatives up to order $k$ are bounded, uniformly continuous and bounded, vanish at infinity, and have compact support, respectively.  Gradient of $f\in\mathcal{C}^1(\R^\n,\R)$ is denoted by $\nabla f(x)=(\partial_1 f(x),\dots,\partial_\n f(x))^\mathrm{T}$, and for $f=(f_1,\dots,f_\m)^\mathrm{T}\in\mathcal{C}^1(\R^\n,\R^\m)$ we write $\mathrm{D} f(x)=(\nabla f_1(x),\dots,\nabla f_\m(x))^\mathrm{T}$  for the corresponding Jacobian. For $\tau=(\tau_1,\dots, \tau_{\n})^{\mathrm{T}}\in (0,\infty)^{\n}$, we let $\ZZ^{\n}_\tau\df\{(\tau_1 k_1,\dotsc,\tau_{\n} k_\n)^\mathrm{T}\colon$\linebreak $ (k_1,\dotsc,k_\n)^\mathrm{T}\in\ZZ^\n\},$
and, for  $x\in\R^{\n}$, 
$$[x]_\tau\,\df\,\bigl\{y\in\R^{\n}\colon x-y\in\ZZ_\tau^{\n}\bigr\}\,,\qquad\textrm{and}\qquad
\mathbb{T}^{\n}_\tau\,\df\,\bigl\{[x]_\tau\colon x\in\R^{\n}\bigr\}\,.$$
Clearly,
$\mathbb{T}^{\n}_\tau$ is obtained
by identifying the opposite
faces of $[0,\tau]\df[0,\tau_1]\times\cdots\times[0,\tau_{\n}]$. 
The corresponding Borel $\sigma$-algebra is denoted by $\mathfrak{B}(\mathbb{T}_\tau^\n)$, which can be identified with the sub-$\sigma$-algebra of $\mathfrak{B}(\R^\n)$  of sets of the form $\bigcup_{k_\tau\in\ZZ_\tau^\n}\{x+k_\tau\colon x\in B_\tau\}$, $B_\tau\in\mathfrak{B}([0,\tau])$.
The covering map $\R^\n\ni x\mapsto [x]_\tau\in\mathbb{T}_\tau^\n$ is denoted by 
$\Pi_{\tau}(x)$.
A
function $f:\R^{\n}\to\R^{\m}$ is called $\tau$-periodic if
$$f(x+k_\tau )\,=\,f(x)\qquad \forall\, (x,k_\tau)\in\R^{\n}\times\ZZ^{\n}_\tau\,.$$
Clearly, every $\tau$-periodic function $f(x)$ is completely and uniquely determined by its restriction $f|_{[0,\tau]}(x)$ to $[0,\tau]$, and since  $f|_{[0,\tau]}(x)$ assumes the same value on opposite faces of $[0,\tau]$ it can be identified  by a function $f_\tau:\mathbb{T}^{\dd}_\tau\to\R^{\m}$ given with $f_\tau ([x]_\tau)\df f(x).$  Using this identification, in an analogous way as above we define  $\mathcal{B}(\mathbb{T}_\tau^\n,\R^\m)$, $\mathcal{B}_b(\mathbb{T}_\tau^\n,\R^\m)$  and $\mathcal{C}_b^k(\mathbb{T}_\tau^\n,\R^\m)=\mathcal{C}^k(\mathbb{T}_\tau^\n,\R^\m)$, $k\in\N_0\cup\{\infty\}$.
For notational convenience, we   write $x$ instead of $[x]_\tau$, and $f(x)$ instead of $f_\tau(x)$.


\subsection{Organization of the article}  In the next section, we first discuss certain structural and ergodic properties of a diffusion process associated to the operator $\mathcal{L}^\varepsilon$. 
Then, in \Cref{S3}, we prove that under an appropriate centering this  process satisfies a functional CLT, as $\varepsilon\to0$,  with a Browninan  limit, which is the key probabilistic argument in discussing the homogenization of the problems in \cref{ES1.1,ES1.2}.
Finally, in \Cref{S4},  we prove the homogenization results.

 
\section{Structural properties of the associated diffusion process}\label{S2}

Throughout the article we impose the following assumptions on the  coefficients $a(x)$, $b(x)$ and $c(x)$:
 
 \medskip
 
\begin{description}
	\item[\ttup{A1}] \begin{enumerate}
		 	\item[\ttup{i}] there is  $\upsigma\in\mathcal{B}(\R^\n,\R^{\n\times\m})$ such that $a(x)=\upsigma(x)\upsigma(x)^{\rm T}$ for all $x\in\R^{\n}$; 
		
		\medskip
		
		\item[\ttup{ii}]
			$\upsigma(x)$, $b(x)$ and $c(x)$ are continuous and $\tau$-periodic; 
		
		\medskip
		
		\item[\ttup{iii}] there is $\Theta>0$ and a non-decreasing concave function $\uptheta:(0,\infty)\to(0,\infty)$ satisfying $$\int_{0+}\frac{\D v}{\uptheta(v)}\,=\,\infty\,,$$ such that for all $x,y\in[0,\tau]$,
		\begin{equation}\label{ES2.1} \max\bigl\{\HS{\upsigma(x)-\upsigma(y)}^2,(b(x)-b(y))(x-y)^\mathrm{T},(c(x)-c(y))(x-y)^\mathrm{T}\bigr\}\,\le\,\Theta\,|x-y|\,\uptheta(|x-y|)\,.
		\end{equation}
	\end{enumerate}
\end{description}
 
 \medskip
 
\noindent According to \cite[Theorems 2.2 and 2.4]{Xi-Zhu-2019}, \ttup{A1}  implies that for any $\varepsilon>0$, $x\in\R^\n$  and a  given standard $m$-dimensional Brownian motion $\{B(t)\}_{t\ge0}$ (defined on a stochastic basis $(\Omega,\mathcal{F},\{\mathcal{F}_t\}_{t\ge0},$\linebreak$\Prob)$ satisfying the usual conditions), the following stochastic differential equation (SDE):
\begin{align*}\D X^\varepsilon(x,t)&\,=\,\bigl(\varepsilon^{-1} b\bigl(X^\varepsilon(x,t)/\varepsilon\bigr)+ c\bigl(X^\varepsilon(x,t)/\varepsilon\bigr)\bigr)\, \D t +\upsigma\bigl(X^\varepsilon(x,t)/\varepsilon\bigr)\,\D B(t)\\X^\varepsilon(x,0)&\,=\,x\in\R^{\n}\,,\end{align*}
admits a unique strong solution $\process{X^\varepsilon}$ which is a conservative (non-explosive) strong Markov process with continuous sample paths, and transition kernel $p^\varepsilon(t,x,\D y)=\Prob\bigl(X^\varepsilon(x,t)\in\D y\bigr)$, $t\ge0$, $x\in\R^{\n}$. Furthermore,  due to \cite[Proposition 4.2]{Xi-Zhu-2019} the process $\process{X^\varepsilon}$ possesses the $\mathcal{C}_b$-Feller  property, that is, $\PP^\varepsilon_tf\in \mathcal{C}_b(\R^\n,\R)$ for any $t\ge0$ and $f\in \mathcal{C}_b(\R^\n,\R)$, where $$\PP^\varepsilon_tf(\cdot)\,\df\,\int_{\R^{\dd}}f(y)\,p^{\varepsilon}(t,\cdot,\D y)\,,\qquad t\ge0\,,\quad f\in\mathcal{B}_b(\R^{\n},\R)\,,$$ stands for the corresponding operator semigroup defined on the Banach space \linebreak $(\mathcal{B}_b(\R^\n,\R), \lVert\cdot\rVert_{\infty})$.
The $\mathcal{B}_b$-infinitesimal generator $(\mathcal{A}^\varepsilon, \mathcal{D}_{\mathcal{A}^\varepsilon})$ of $\{\PP^\varepsilon_t\}_{t\ge0}$ (or of  $\process{X^\varepsilon}$) is a linear operator $\mathcal{A}^\varepsilon:\mathcal{D}_{\mathcal{A}^\varepsilon} \to \mathcal{B}_b(\R^\n,\R)$ defined by
\begin{equation*}
\mathcal{A}^\varepsilon f\,\df\,\lim_{t \to 0} \frac{\PP^\varepsilon_t f-f}{t}\,,\qquad f \in \mathcal{D}_{\mathcal{A}^\varepsilon}\,\df\,\left\{f \in \mathcal{B}_b(\R^\n,\R)\colon \lim_{t \to 0} \frac{\PP^\varepsilon_t f-f}{t}\ \text{ exists in}\ \lVert\cdot\rVert_{\infty} \right\}\,.
\end{equation*}
 By employing It\^{o}'s formula we easily see that $$\lim_{t\to0}\left\lVert\bigl(\PP^\varepsilon_tf-f\bigr)/t-\mathcal{L}^\varepsilon f\right\rVert_\infty\,=\,0\qquad \forall\,f\in\mathcal{C}_{u,b}^2(\R^\n,\R)\,,$$
that is, $\mathcal{C}_{u,b}^2(\R^\n,\R)\subseteq \mathcal{D}_{\mathcal{A}^\varepsilon}$ and $\A^\varepsilon|_{\mathcal{C}_{u,b}^2(\R^\n,\R)}=\mathcal{L}^\varepsilon.$

Following \cite{Freidlin-1964} (see also \cite[Lemma 3.4.1]{Bensoussan-Lions-Papanicolaou-Book-1978}), for $\varepsilon>0$  let $\bar X^\varepsilon(x,t)\df \varepsilon^{-1}X^\varepsilon(\varepsilon x,\varepsilon^2 t)$, $t\ge0$. Clearly, $\process{\bar X^\varepsilon}$ satisfies \begin{equation}\label{ES2.2}\begin{aligned}\D \bar X^\varepsilon(x,t)&\,=\, \bigl(b\bigl(\bar X^\varepsilon(x,t)\bigr)+\varepsilon c\bigl(\bar X^\varepsilon(x,t)\bigr)\bigr)\,\D t+\upsigma\bigl(\bar X^\varepsilon(x,t)\bigr)\,\D B^\varepsilon(t)\\
\bar X^\varepsilon(x,0)&\,=\,x\in\R^\n\,,\end{aligned}\end{equation} where $B^\varepsilon(t)\df\varepsilon^{-1}B(\varepsilon^2t)$, $t\ge0$. Observe that $\{B^\varepsilon(t)\}_{t\ge0}\stackrel{(\rm{d})}{=}\{B(t)\}_{t\ge0}$, although it is not a martingale with respect to $\{\mathcal{F}_t\}_{t\ge0}$.  Here, $\stackrel{(\rm{d})}{=}$ denotes the equality in distribution. Let also $\process{\bar X^0}$ be a solution to \begin{equation}\begin{aligned}\label{ES2.3}\D \bar X^0(x,t)&\,=\, b\bigl(\bar X^0(x,t)\bigr)\,\D t+\upsigma\bigl(\bar X^0(x,t)\bigr)\,\D B(t)\\
 \bar X^0(x,0)&\,=\,x\in\R^\n\,.\end{aligned}\end{equation} Clearly, the processes $\process{\bar X^\varepsilon}$, $\varepsilon\ge0$, share the same structural properties as $\process{ X^\varepsilon}$, $\varepsilon>0$, mentioned above.
Denote by $\bar p^\varepsilon(t,x,\D y)=\Prob\bigl(\bar X^\varepsilon(x,t)\in\D y\bigr)$, $t\ge0$, $x\in\R^\n$,  $\{\mathcal{\bar P}^\varepsilon_t\}_{t\ge0}$ and  $(\bar{\mathcal{A}}^\varepsilon,\mathcal{D}_{\bar{\mathcal{A}}^\varepsilon})$ the corresponding transition kernel, operator semigroup and  $\mathcal{B}_b$-infinitesimal generator,   respectively. From \cite[Theorem IX.4.8]{Jacod-Shiryaev-2003} it follows that \begin{equation}\label{ES2.4}\process{\bar X^\varepsilon}\,\xRightarrow[\varepsilon\to0]{({\rm d})}\,\process{\bar X^0}\,.\end{equation} In particular, for any $t\ge0$ and $x\in\R^\n$, $\bar p^\varepsilon(t,x,\D y)\xRightarrow[\varepsilon\to0]{({\rm w})}\bar p^0(t,x,\D y)$, that is, $$\lim_{\varepsilon \to 0}\bar\PP_t^\varepsilon f(x)\,=\, \bar\PP_t^0 f(x)$$ for any $t\ge0$, $x\in\R^\n$ and $f\in\mathcal{C}_b(\R^\n,\R)$. Here,      $\ \xRightarrow[]{({\rm d})}$ denotes the convergence in the space of continuous functions endowed with the locally uniform topology (see \cite[Chapter VI]{Jacod-Shiryaev-2003} for details), and   $\ \xRightarrow[]{({\rm w})}$
  stands for the weak convergence of probability measures.

Next, observe that due to $\tau$-periodicity of the coefficients $\{\bar X^\varepsilon(x+k_{\tau},t)\}_{t\ge0}$ and $\{\bar X^\varepsilon(x,t)+k_{\tau}\}_{t\ge0}$, $\varepsilon\ge0$, $x\in\R^\n$ , $k_{\tau} \in\ZZ^\n_{\tau}$,  are indistinguishable. In particular, $$\bar p^\varepsilon(t,x+k_{\tau},B)\,=\,\bar p^\varepsilon(t,x,B-k_{\tau})$$ for all $\varepsilon\ge0$, $t\ge0$, $x\in\R^\n$, $k_{\tau}\in\ZZ^\n_{\tau}$ and $B\in\mathfrak{B}(\R^\n),$
which implies that $\{\bar\PP^\varepsilon_t\}_{t\ge0}$ preserves the class of $\tau$-periodic functions in $\mathcal{B}_b(\R^\n,\R)$. Thus,  according to \cite[Proposition 3.8.3]{Kolokoltsov-Book-2011}   the projection of  $\process{\bar X^{\varepsilon}}$ with respect to $\Pi_{\tau}(x)$ on the torus $\mathbb{T}^\n_{\tau}$, denoted by   $\process{\bar X^{\varepsilon,\tau}}$, is a Markov process on $(\mathbb{T}^\n_{\tau},\mathfrak{B}(\mathbb{T}^\n_{\tau}))$ with transition kernel  given by
\begin{equation}\label{ES2.5}
\bar p^{\varepsilon,\tau}(t,x,B)\,=\,\bar p^\varepsilon\bigl(t,z_x,\Pi_{\tau}^{-1}(B)\bigr)\end{equation} for $\varepsilon\ge0$, $ t\ge 0$, $x\in \mathbb{T}^\n_{\tau}$, $B \in \mathfrak{B}(\mathbb{T}^\n_{\tau})$ and $z_x \in \Pi_{\tau}^{-1}(\{x\})$.
In particular, 
$\process{\bar X^{\varepsilon,\tau}}$ is  a $\mathcal{C}_b$-Feller  process.

\medskip

\begin{proposition}\label{P2.1} Under \ttup{A1}, for any $t\ge0$ and $\tau$-periodic $f\in\mathcal{C}_b(\R^\n,\R)$ it holds that $$\lim_{\varepsilon \to 0}\|\bar{\mathcal{P}}_t^\varepsilon f-\bar{\mathcal{P}}_t^0 f\|_\infty\,=\,0\,.$$
\end{proposition}
\begin{proof} From \cref{ES2.4,ES2.5} we see that
	 $$\process{\bar X^{\varepsilon,\tau}}\,\xRightarrow[\varepsilon\to0]{({\rm d})}\,\process{\bar X^{0,\tau}}\,.$$
	Now, 
	since $\mathbb{T}^\dd_{\tau}$  is compact,   the assertion  follows from \cite[Theorem 17.25]{Kallenberg-Book-1997}.
\end{proof}

\medskip

Clearly, for any $x\in\R^\n$ the matrix $a(x)$ is symmetric and non-negative definite. 
We further assume     

\medskip

\begin{description}
	\item[\ttup{A2}] 
	\begin{itemize}
		\item [\ttup{i}] there is an open connected set $\mathscr{O}\subset \left[0,\tau\right]$ such that  the matrix $a(x)$ is positive definite on $\overline{\mathscr{O}}$, that is,  $$ \xi^{\mathrm{T}}a(x)\,\xi \,>\,0\qquad\forall\, (x,\xi)\in \overline{\mathscr{O}}\times\R^{\n}\setminus \{0\}\,;$$
		
		\medskip

		\item[\ttup{ii}] $a(x)$, $b(x)$ and $c(x)$  are $\gamma$-H\"{o}lder continuous for some $0<\gamma\le1$, that is, there is $\Gamma>0$ such that \begin{equation}\label{ES2.6}\HS{a(x)-a(y)}+|b(x)-b(y)|+|c(x)-c(y)|\,\le\,\Gamma\,|x-y|^\gamma\qquad \forall\,x,y\in\R^\n\,.\end{equation}
		
	\end{itemize}
	
\end{description}

\medskip

	\begin{remark} \label{R2.2}
		\begin{itemize}
			\item [\ttup{i}] For a given symmetric, non-negative definite and Borel measurable $n\times n$-matrix-valued function $a(x)$ there is a unique non-negative definite and Borel measurable $n\times n$-matrix-valued function $\bar{\upsigma}(x)$ such that $a(x)=\bar{\upsigma}(x)\bar{\upsigma}(x)^{\rm T}$ for all $x\in\R^\n$. In general, it is not clear that smoothness (H\"{o}lder continuity or differentiability) of $a(x)$  implies smoothness of $\bar{\upsigma}(x)$. However, if $a(x)$ is additionally positive definite or twice continuously differentiable this will be the case (see \cite[Lemma 6.1.1 and Theorem 6.1.2]{Friedman-Book-1975}). In particular, under \ttup{A2},  $\bar{\upsigma}(x)$ will be $\gamma$-H\"{o}lder continuous on  $\mathscr{O}$.
			
			\medskip
			
			\item[\ttup{ii}] \Cref{ES2.1} holds true if for all $x,y\in\R^\n$,	\begin{equation}\label{ER2.2} \HS{\upsigma(x)-\upsigma(y)}^2+\, |x-y|\bigl(|b(x)-b(y)|+|c(x)-c(y)|\bigr)\,\le\,\Theta\,|x-y|\,\theta(|x-y|)\,.
			\end{equation} 	Clearly, \cref{ER2.2}, together with periodicity of $\upsigma(x)$, automatically implies  $1/2$-H\"{o}lder continuity of $\upsigma(x)$. Moreover, since \begin{align*}\HS{a(x)-a(y)}&\,=\,\HS{\upsigma(x)\upsigma(x)^{\rm T}-\upsigma(y)\upsigma(y)^{\rm T}}\\&\,\le\,\HS{(\upsigma(x)-\upsigma(y))\,\upsigma(x)^{\rm T}}+\HS{\upsigma(y)\,(\upsigma(x)^{\rm T}-\upsigma(y)^{\rm T})}\\
			&\,\le\,2\ \lVert\upsigma\rVert_\infty\,\HS{\upsigma(x)-\upsigma(y)}\end{align*}
			it also implies $1/2$-H\"{o}lder continuity of $a(x)$.
			In addition, if $\limsup_{v\to0}\uptheta(v)/v^{\gamma}<\infty$ for some $\gamma\in(0,1]$, it is easy to see that \cref{ER2.2} implies $\gamma$-H\"{o}lder continuity of $b(x)$ and $c(x)$, and  $(1+\gamma)/2$-H\"{o}lder continuity of $\upsigma(x)$ and $a(x)$.
			
			\medskip
			
			\item[\ttup{iii}] Assumptions \ttup{A1}-\ttup{A2} imply  that  $a(x)$ is  uniformly elliptic on $\overline{\mathscr{O}}$, that is,  there is  $\alpha >0$ such that $$\xi^{\mathrm{T}}a(x)\,\xi\, \ge\, \alpha |\xi| ^2\qquad\forall\, (x,\xi)\in \overline{\mathscr{O}}\times\R^{\n}\,.$$ Indeed, since for every $x \in \overline{\mathscr{O}}$ the matrix $a(x)$ is symmetric and positive definite, the corresponding  eigenvalues $\uplambda_1(x),\dots, \uplambda_\n(x)$ are real and positive. Also, 
			since  $\uplambda_1(x),\dots, \uplambda_\n(x)$ are  roots of the polynomial $\lambda\mapsto\det(a(x)- \lambda\Id_\n)$ we see that each $\uplambda_i\colon \overline{\mathscr{O}}\to(0,\infty)$ is continuous. Here, $\Id_\n$ stands for the $n\times n$-identity matrix. Hence, due to compactness of $\overline{\mathscr{O}}$, we conclude that there is $\alpha>0$ such that $a(x)-\alpha \Id_\n$ is positive definite on $\overline{\mathscr{O}}$, which proves the assertion.
		\end{itemize}
	\end{remark}

\medskip

For $\varepsilon\ge0$, $x\in\R^\n$ and $B\in\mathfrak{B}(\R^\n)$, let $\bar\uptau^{\varepsilon,x}_B\df\inf\{t\ge0\colon \bar X^\varepsilon(x,t)\in B\}$ be the first entry time of  $B$ by $\process{\bar X^\varepsilon}$.
Assume 

\medskip

\begin{description}
	\item[\ttup{A3}] there is $\varepsilon_0>0$ such that  $$\Prob\bigl(\bar\uptau^{\varepsilon,x}_{\mathscr{O}+\tau}<\infty\bigr)\,>\,0\qquad\forall\, (\varepsilon,x)\in[0,\varepsilon_0]\times\R^\n\,,$$ where $\mathscr{O}+\tau\df\{x+k_\tau \colon  x\in \mathscr{O},\ k_\tau \in\ZZ^\n_\tau \}$.  
\end{description}

\medskip 

\noindent
Under \ttup{A1},    \cite[Theorem 3.1]{Meyn-Tweedie-AdvAP-II-1993} shows that $\process{\bar X^{\varepsilon,\tau}}$, $\varepsilon\ge0$, admits at least one invariant probability measure. Assuming additionally  \ttup{A2}-\ttup{A3}, in what follows we show that \linebreak $\process{\bar X^{\varepsilon,\tau}}$, $\varepsilon\in[0,\varepsilon_0]$, 
admits one, and only one, invariant measure, and the corresponding marginals  converge as $t\to\infty$  to the  invariant measure in the total variation norm (with exponential rate).

\medskip

 \begin{proposition}\label{P2.3}
 	Under \ttup{A1}-\ttup{A3}, there exists a  measure $\uppsi(\D x)$ on $\mathbb{T}_\tau^\n$ such that
 	
 	\medskip
 	
 	\begin{itemize}
 		\item [\ttup{i}] $\mathrm {supp}(\uppsi)$ has nonempty interior;
 		
 		\medskip
 		
 		\item [\ttup{ii}] for every $x \in \mathbb{T}_\tau^\n$ and $\varepsilon\in[0,\varepsilon_0]$ there is $t_{x,\varepsilon}\geq 0$ such that
 		$$\uppsi(B)\,>\,0\,\Longrightarrow\,\bar p^{\varepsilon,\tau}(t,x,B)\,>\,0\qquad  \forall\, t\in[t_{x,\varepsilon},\infty)\,.
 		$$
 	\end{itemize}
 	\end{proposition}
 
 \begin{proof}
 	According to \cite[Theorems 7.3.6 and 7.3.7]{Durrett-Book-1996}  there is a strictly positive function $q^\varepsilon(t,x,y)$ on $(0,\infty)\times \overline{\mathscr{O}}+\tau\times \overline{\mathscr{O}}+\tau$, jointly continuous in $t$, $x$ and $y$,  
 		satisfying
 		$$
 		 \mathbb{E} \left[f(\bar X^\varepsilon(x,t))\,\mathbb{1}_{(t,\infty]}\bigl(\bar\uptau^{\varepsilon,x}_{(\overline {\mathscr{O}}+\tau)^c}\bigr)\right]\,=\,\int_{\overline {\mathscr{O}}+\tau}f(y)\,q^\varepsilon(t,x,y) \,\D y$$ for all $t>0$, $x \in \overline {\mathscr{O}}+\tau$ and  $f\in\mathcal{C}_b(\R^\n,\R)$.
 	By employing dominated convergence theorem it is straightforward to check that
the above relation holds also for $\mathbb{1}_{\mathscr{U}+\tau}(x)$ for any open set $\mathscr{U}\subseteq\mathscr{O}$.
 	Denote by $\mathcal{D}$ the class of all $B \in \mathfrak{B}(\mathscr{O})$  such that 
 	$$ 
 	\Prob \bigl(\bar X^\varepsilon(x,t) \in B+\tau\,,\, \bar\uptau^{\varepsilon,x}_{(\overline{\mathscr{O}}+\tau)^c}>t \bigr)\,=\,\int_{B+\tau}q^\varepsilon(t,x,y)\, \D y\,.
 	$$
 	Obviously, $\mathcal{D}$ contains the $\pi$-system of open rectangles in $\mathfrak{B}(\mathscr{O})$, and forms a $\lambda$-system. Thus, Dynkin's $\pi$-$\lambda$ theorem implies that $\mathcal{D}=\mathfrak{B}(\mathscr{O})$, and 
 	for all $t>0$, $x \in \overline{\mathscr{O}}+\tau$ and $B \in \mathfrak{B}([0,\tau])$  we have
 	$$
 	\bar p^\varepsilon(t,x,B+\tau)\,\ge\, \int_{(B\cap \mathscr{O})+\tau}q^\varepsilon (t,x,y)\, \D y\,.
 	$$
 	Set now $\overline\uppsi (B+\tau)\df\uplambda ((B\cap \mathscr{O})+\tau)$, $B\in\mathfrak{B}([0,\tau])$,
 	where $\uplambda(\D x)$ stands for the Lebesgue measure on $\R^\n$. Clearly, by construction, 
 	$\overline\uppsi(\D x)$ is a measure on $\sigma$-algebra $\mathfrak{B}([0,\tau])+\tau\df\{B+\tau\colon B\in\mathfrak{B}([0,\tau])\}$, $\mathrm{supp}(\overline\uppsi)$ has non-empty interior, and 
 	 for $B \in \mathfrak{B}([0,\tau])$ it holds that
 	\begin{equation}\label{EP2.3}
 	\overline\uppsi (B+\tau)\,>\,0\, \Longrightarrow\, \bar p^\varepsilon (t,x,B+\tau) \,>\,0 \qquad \forall \,(t,x)\in (0,\infty)\times\overline{\mathscr{O}}+\tau \,. 
 \end{equation}
 To this end, it remains to show that for each $x\in([0,\tau]\setminus\overline{\mathscr{O}})+\tau$  there is $t_{x,\varepsilon}\ge0$ such that the implication in \cref{EP2.3} holds   for all $t\ge t_{x,\varepsilon}$. Since $\process{\bar X^\varepsilon}$ has continuous sample paths and $\mathscr{O}$ is an open set, we have that  
 	\begin{align*}
 	\sum_{t \in \mathbb{Q}_+} \bar p^\varepsilon(t,x,\mathscr{O}+\tau)&\,\geq\, \Prob\bigl(\exists t \in \mathbb{Q}_+ \text{ such that }\bar X^\varepsilon(x,t)\in \mathscr{O}+\tau\bigl)\\
 	&\,=\,
 	\Prob \bigl(\exists t \geq 0 \text{ such that }\bar X^\varepsilon(x,t) \in \mathscr{O}+\tau\bigr)\\
 	&\,=\,\Prob\bigl(\bar\uptau^{\varepsilon,x}_{\mathscr{O}+\tau}<\infty\bigr)\,.
 	\end{align*}
 	From \ttup{A3} we see that there is $t_{x,\varepsilon}\in\mathbb{Q}_+$ such that $\bar p^\varepsilon(t_{x,\varepsilon},x,\mathscr{O}+\tau)>0$. Let $B \in \mathfrak{B}([0,\tau])$ be such that $\overline\uppsi(B+\tau)>0$. For any $t> t_{x,\varepsilon}$ we then have
 	\begin{align*}
 	\bar p^\varepsilon(t,x,B+\tau)&\,\ge\, \Prob\bigl(\bar X^\varepsilon(x,t)\in B+\tau\,,\, \bar X^\varepsilon(x,t_{x,\varepsilon}) \in \mathscr{O}+\tau\bigr)\\&\,=\,\int_{\mathscr{O}+\tau} \bar p^\varepsilon(t-t_{x,\varepsilon},y,B+\tau)\,\bar p^\varepsilon(t_{x,\varepsilon},x,\D y)\,,
 \end{align*}
 	which is strictly positive because of \cref{EP2.3}. The result now follows by setting $\uppsi(B):=\overline\uppsi(\Pi_{\tau}^{-1}(B))$, $B\in\mathfrak{B}(\mathbb{T}_\tau^\n)$, and using \cref{ES2.5}.
 	\end{proof}
 
 \medskip
 
 From \Cref{P2.3} it immediately follows that $\process{\bar X^{\varepsilon,\tau}}$ is  irreducible in the sense of \cite{Down-Meyn-Tweedie-1995}, that is, $\uppsi(B)>0$ implies that $\int_0^\infty\bar p^{\varepsilon,\tau}(t,x,B)\,\D t>0$ for all $x\in\mathbb{T}_\tau^\n$. This automatically entails that $\process{\bar X^{\varepsilon,\tau}}$ admits one, and only one, invariant probability measure $\uppi^\varepsilon(\D x)$. Namely, 
 according to \cite[Theorem 2.3]{Tweedie-1994} every irreducible Markov process is either  transient or recurrent. Due to the fact that   $\process{\bar X^{\varepsilon,\tau}}$  admits at least one invariant probability measure it clearly cannot be transient. The assertion now follows  from \cite[Theorem 2.6]{Tweedie-1994} which  states that every recurrent Markov process admits a unique (up to constant multiplies) invariant measure.

 \medskip

 \begin{proposition} \label{P2.4}Under \ttup{A1}-\ttup{A3},
 	there are $\gamma>0$ and $\varGamma>0$, such that $$\sup_{x\in\mathbb{T}_\tau^\n}\|\bar p^{\varepsilon,\tau}(t,x,\D y)-\uppi^\varepsilon(\D y)\|_{\mathrm{TV}}\,\le\,\varGamma\E^{-\gamma t}\qquad \forall\,(\varepsilon,t)\in[0,\varepsilon_0]\times[0,\infty)\,,$$ where $\lVert\cdot\rVert_{\mathrm{TV}}$ denotes the total variation norm on the space of
 	signed measures on $\mathfrak{B}(\mathbb{T}_\tau^\n)$.
 		\end{proposition}
 \begin{proof} First, \cite[Theorems 5.1 and 7.1]{Tweedie-1994} together with the $\mathcal{C}_b$-Feller property of $\process{\bar X^{\varepsilon,\tau}}$ and \Cref{P2.3} imply that $\mathbb{T}_\tau^\n$ is a petite set for  $\process{\bar X^{\varepsilon,\tau}}$ (see \cite{Tweedie-1994} for  the definition of petite sets). Next, from \cite[Theorem 4.2]{Meyn-Tweedie-AdvAP-III-1993} (with $c=d=1$, $C=\mathbb{T}_\tau^\n$, $f(x)=V(x)\equiv1$, and $\mathcal{A}V(x)\equiv0$), \Cref{P2.3} (which implies that  $\sum_{i=1}^\infty\bar p^{\varepsilon,\tau}(i,x,B)>0$ for all $x\in\mathbb{T}_\tau^\n$ whenever $\uppsi(B)>0$), and \cite[Proposition 6.1]{Meyn-Tweedie-AdvAP-II-1993} we see that $\process{\bar X^{\varepsilon,\tau}}$ is aperiodic in the sense of \cite{Down-Meyn-Tweedie-1995}. The desired result now follows from \cite[Theorem 5.2]{Down-Meyn-Tweedie-1995} by taking $c=b=1$, $C=\mathbb{T}_\tau^\n$, $\tilde V(x)\equiv1$, and $\tilde{\mathcal{A}}\tilde V(x)\equiv0$.
 	\end{proof}
 
 \medskip
 
 From \Cref{P2.4} we see that for any $f\in\mathcal{B}_b(\mathbb{T}_\tau^\n,\R)$ satisfying $\uppi^\varepsilon(f)=0$, $\varepsilon\in[0,\varepsilon_0]$, it holds that \begin{equation}\label{ES2.7}\|\bar{\mathcal{P}}_t^{\varepsilon,\tau}f\|_\infty\,\le\,\varGamma\|f\|_\infty\E^{-\gamma t} \qquad \forall\,  t\ge0\,.\end{equation}
 
 \medskip
 
 \begin{proposition}\label{P2.5}Under \ttup{A1}-\ttup{A3},
 	$$\uppi^\varepsilon(\D x)\,\xRightarrow[\varepsilon\to0]{({\rm w})}\,\uppi^0(\D x)\,.$$ 
 		\end{proposition}
 	\begin{proof}
 		Since $\mathbb{T}_\tau^\n$ is compact the family of probability measures $\{\uppi^\varepsilon(\D x)\}_{\varepsilon\ge0}$ is  tight. Hence, for any sequence $\{\varepsilon_i\}_{i\in\N}\subset[0,\varepsilon_0]$ converging to $0$ there is a further subsequence $\{\varepsilon_{i_j}\}_{j\in\N}$ such that $\{\uppi^{\varepsilon_{i_j}}(\D x)\}_{j\in\N}$ converges weakly to some probability measure $\bar\uppi^0(\D x)$. Take $f\in\mathcal{C}(\mathbb{T}_\tau^\n,\R)$, and fix $t\ge0$ and $\epsilon>0$. From \Cref{P2.1} we have that there is $0<\varepsilon_1\le\varepsilon_0$ such that $$\|\bar{\mathcal{P}}^{\varepsilon,\tau}_tf-\bar{\mathcal{P}}^{0,\tau}_tf\|_\infty\,\le\,\epsilon\qquad \forall\,\varepsilon\in[0,\varepsilon_1]\,.$$ 
 		We now have that
 		\begin{align*}
 		|\bar\uppi^{0}(f)-\bar\uppi^0\bigl(\bar{\mathcal{P}}_t^{0,\tau}f\bigr)|&\,=\,\lim_{j\to\infty}|\uppi^{\varepsilon_{i_j}}(f)-\bar\uppi^0\bigl(\bar{\mathcal{P}}_t^{0,\tau}f\bigr)|\\&\,=\,\lim_{j\to\infty}|\bar\uppi^{\varepsilon_{i_j}}\bigl(\bar{\mathcal{P}}_t^{\varepsilon_{i_j},\tau}f\bigr)-\bar\uppi^0\bigl(\bar{\mathcal{P}}_t^{0,\tau}f\bigr)|\\&\le\,
 		\limsup_{j\to\infty}|\uppi^{\varepsilon_{i_j}}\bigl(\bar{\mathcal{P}}_t^{\varepsilon_{i_j},\tau}f\bigr)-\bar\uppi^{\varepsilon_{i_j}}\bigl(\bar{\mathcal{P}}_t^{0,\tau}f\bigr)| +\lim_{j\to\infty}|\uppi^{\varepsilon_{i_j}}\bigl(\bar{\mathcal{P}}_t^{0,\tau}f\bigr)-\bar\uppi^0\bigl(\bar{\mathcal{P}}_t^{0,\tau}f\bigr)|\\
 		&\,\le\,\epsilon\,,
 		\end{align*}
 	which implies that  $\bar\uppi^0(\D x)$ is an invariant probability measure for $\process{X^{0,\tau}}$. Thus,  $\bar\uppi^0(\D x)=\uppi^0(\D x)$, which proves the  assertion.
 \end{proof}

 
\section{CLT for the process $\process{X^{\varepsilon}}$}
\label{S3}

Under the assumption that $a(x)$ is uniformly elliptic  (that is, $\mathscr{O}=(0,\tau_1)\times\cdots\times(0,\tau_\dd)$) and twice continuously differentiable,  and that  $b,c\in\mathcal{C}^2(\R^\n,\R^n)$ (in particular \ttup{A1}-\ttup{A3} are automatically satisfied with $\theta(v)=v$, $\gamma=1$ and $\mathscr{O}=(0,\tau_1)\times\cdots\times(0,\tau_\dd)$), in \cite[Theorem 3.4.4]{Bensoussan-Lions-Papanicolaou-Book-1978} 
it has been shown that: (i) the equation $\bar{\mathcal{A}}^0\beta(x)=b(x)-\uppi^0(b)$ admits a unique $\tau$-periodic solution $\beta\in C^2(\R^\n,\R^\n)$,
and (ii) it holds that \begin{equation}\label{ES3.1}\{ X^{\varepsilon}(x,t)-\varepsilon^{-1}\uppi^0(b)t \}_{t\ge0}\,\xRightarrow[\varepsilon\to0]{({\rm d})}\,\{W^{\mathsf{a},\mathsf{b}}(x,t)\}_{t\ge0}\,,\end{equation} where $\{W^{\mathsf{a},\mathsf{b}}(x,t)\}_{t\ge0}$ is a  $n$-dimensional   Brownian motion   determined by covariance matrix and drift vector   \begin{equation}\label{ES3.2} \mathsf{a}\,=\,\uppi^0\bigl((\Id_\n-\DD\beta)\,a\,(\Id_\n-\DD\beta)^{\mathrm{T}}\bigr)  \qquad\text{and}\qquad \mathsf{b}\,=\,\uppi^0\bigl((\Id_\n-\DD\beta)\,c\bigr)\,,\end{equation} respectively.
In this section, we derive an analogous results for the case when $a(x)$ is not necessarily uniformly elliptic.   From \ttup{A1}-\ttup{A2}, that is, from \cref{ES2.7}, we see that the function $$x\,\mapsto\,-\int_0^\infty\bar\PP_t^0\bigl(b-\uppi^0(b)\bigr)(x)\,\D t\,,\qquad x\in\R^\n\,,$$ (which we again denote by $\beta(x)$) is well defined, $\tau$-periodic, continuous, and satisfies $\beta\in\mathcal{D}_{\bar{\mathcal{A}}^0}$ and $\bar{\mathcal{A}}^0\beta(x)=b(x)-\uppi^0(b)$.  Note that under the uniform ellipticity (and smoothness) assumption this function coincides with the function $\beta(x)$ discussed above. A crucial step in the proof of \cref{ES3.1} in the uniformly elliptic case is an application of It\^{o}'s formula to the process $\{\beta(X^{\varepsilon}(x,t))\}_{t\ge0}$ (recall that in this case $\beta\in C^2(\R^\n,\R^\n)$). On the other hand, in the case when the coefficient $a(x)$ can be degenerate it is not clear how to conclude necessary smoothness of $\beta(x)$.

\medskip

\subsection{The case  $c(x)\equiv0$}\label{S3.1}
Observe first that  \begin{equation}\label{ES3.3}\process{X^{\varepsilon}}=\{\varepsilon \bar X^{\varepsilon}(x/\varepsilon,t/\varepsilon^2)\}_{t\ge0} \stackrel{(\rm{d})}{=}\{\varepsilon \bar X^{0}(x/\varepsilon,t/\varepsilon^2)\}_{t\ge0}\,.\end{equation}
We  now conclude the following.

\medskip

\begin{theorem} \label{T3.1} Under \ttup{A1}-\ttup{A3},
	the relation in \cref{ES3.1} holds with $$\mathsf{a}\,=\,\uppi^0\bigl(a-\bar a-\bar a^\mathrm{T}-\beta\,\bigl(\bar{\mathcal{A}}^{0,\tau}\beta\bigl)^{\mathrm{T}}-\bigl(\bar{\mathcal{A}}^{0,\tau}\beta\bigr)\,\beta^{\mathrm{T}}\bigr)\qquad \text{and}\qquad \mathsf{b}\,=\,0\,,$$  where   $\bar a\in\mathcal{B}(\R^\n,\R^{\n\times\n})$ is $\tau$-periodic and such that $\uppi^0(\HS{\bar a})<\infty$.
\end{theorem}
\begin{proof}
	We have
	\begin{equation}\label{EPT3.1}
	\begin{aligned}
	& X^\varepsilon(x,t)-x-\varepsilon^{-1}\uppi^0(b)t-\varepsilon\beta\bigl(X^\varepsilon(x,t)/\varepsilon\bigr)+\varepsilon\beta(x/\varepsilon)\\
	&
	\,\stackrel{(\rm{d})}{=}\,
\varepsilon \bar X^0(x/\varepsilon,t/\varepsilon^2)-x-\varepsilon^{-1}\uppi^0(b)t-\varepsilon\beta\bigl(\bar X^0(x/\varepsilon,t/\varepsilon^2)\bigr)+\varepsilon\beta(x/\varepsilon)\\
	&
	\,=\,
		-\varepsilon\beta\bigl(\bar X^0(x/\varepsilon,t/\varepsilon^2)\bigr)+\varepsilon\beta(x/\varepsilon)+\varepsilon\int_0^{\varepsilon^{-2}t}\bigl(b\bigl(\bar X^0(x/\varepsilon,s)\bigr)-\uppi^0(b)\bigr)\,\D s\\
	&\ \ \ \ \ +\varepsilon\int_0^{\varepsilon^{-2}t}\upsigma\bigl(\bar X^0(x/\varepsilon,s)\bigr)\,\D B(s)\qquad \forall\, t\ge0\,.
	\end{aligned}
	\end{equation}
	Due to boundedness of $\beta(x)$, $\{ X^{\varepsilon}( x,t)-x-\varepsilon^{-1} \uppi^0(b)t\}_{t\ge0}$ converges in law if, and only if,
	$\{X^{\varepsilon}( x,t)-x-\varepsilon^{-1}\uppi^0(b)t-\varepsilon\beta(X^\varepsilon(x,t)/\varepsilon)+\varepsilon\beta(x/\varepsilon)\}_{t\ge0}$    converges, and if this is the case the limit is the same. 
	Denote \begin{align*}
	M_1(x,t)&\,\df\,\beta\bigl(\bar X^0(x,t)\bigr)-\beta(x)-\int_0^{t}\bigl(b\bigl(\bar X^0(x,s)\bigr)-\uppi^0(b)\bigr)\,\D s\,,\qquad t\ge0\,,\\
	M_2(x,t)&\,\df\, \int_0^{t}\upsigma\bigl(\bar X^0(x,s)\bigr)\,\D B(s)\,,\qquad t\ge0\,.
	\end{align*}
According to  \cite[Proposition 4.1.7]{Ethier-Kurtz-Book-1986} the processes
$\process{M_k}$, $k=1,2$, are $\{\mathcal{F}_t\}_{t\ge0}$-martingales. Hence,
 $\{X^{\varepsilon}( x,t)-x-\varepsilon^{-1}\uppi^0(b)t-\varepsilon\beta(X^\varepsilon(x,t)/\varepsilon)+\varepsilon\beta(x)\}_{t\ge0}$ is also a $\{\mathcal{F}_t\}_{t\ge0}$-martingale. Next, let $\{\theta_t\}_{t\ge0}$ be the family of  shift operators on  $(\Omega,\mathcal{F},\{\mathcal{F}_t\}_{t\ge0})$ satisfying $\theta_s \circ\bar X^0(x,t)=\bar X^0(x,s+t)$ for all $s,t\ge0$  (see \cite[p. 119]{Oksendal-Book-2003}). 
 We have that $$M_k(x,s+t)\,=\,M_k(x,t)+\theta_t\circ M_k(x,s)\,,\qquad s,t\ge0\,,\quad k=1,2\,.$$ In other words,  the processes
 $\process{M_k}$, $k=1,2$, are continuous additive martingales with respect to $\process{\bar X^0}$, in the sense of \cite{Cinlar-Jacod-1981}.
 Observe that $$M_2(x,t)\,=\,\bar X^0(x,t)-x-\int_0^tb\bigl(\bar X^0(x,s)\bigr)\,\D s\,,\qquad t\ge0\,.$$
 	According to \cite[Theorem VIII.2.17]{Jacod-Shiryaev-2003}, in order to conclude \cref{ES3.1} it suffices to show that \begin{equation}\label{EPT3.1A}\varepsilon^2\langle-M_1(x/\varepsilon,\cdot)+M_2(x/\varepsilon,\cdot),-M_1(x/\varepsilon,\cdot)+M_2(x/\varepsilon,\cdot)\rangle_{t/\varepsilon^2}\,\xrightarrow[\varepsilon\to0]{\Prob}\,\mathsf{a}\, t\,,\qquad t\ge0\,.\end{equation} Here, $\ \xrightarrow[]{\Prob}$ stands for the convergence in probability, and 
 	for two locally square-integrable martingales $\{M(t)\}_{t\ge0}$ and   $\{N(t)\}_{t\ge0}$, $\{\langle M(\cdot),N(\cdot)\rangle_t\}_{t\ge0}$ denotes the corresponding predictable quadratic covariation process and we write $\{\langle M(\cdot)\rangle_t\}_{t\ge0}$ instead of $\{\langle M(\cdot),M(\cdot)\rangle_t\}_{t\ge0}$.
 
	For $i,j=1,\dots,\dd$ and $k,l=1,2$, we have that 
	\begin{equation*}
		\langle M_k^i(x,\cdot),M_l^j(x,\cdot)\rangle_{t}\,=\, 4^{-1}\bigl(\langle M_k^i(x,\cdot)+M_l^j(x,\cdot)\rangle_{t}-\langle M_k^i(x,\cdot)-M_l^j(x,\cdot)\rangle_{t}\bigr)\,,\qquad t\ge0\,.
	\end{equation*} 
	Next, form the martingale representation theorem we see that for each $i,j=1,\dots\dd$ and $k,l=1,2$ it holds that $\D\langle M^i_k(x,\cdot)\pm M^j_l(x,\cdot)\rangle_t\ll \D t$. Thus, due to $\tau$-periodicity of the coefficients and the fact that $\{\bar X^0(x+k_\tau,t)\}_{t\ge0}$ and $\{\bar X^0(x,t)+k_\tau\}_{t\ge0}$ are indistinguishable for all $x\in\R^\dd$ and $k_\tau\in\ZZ_\tau^\dd$,  \cite[Proposition 3.56]{Cinlar-Jacod-Protter-Sharpe-1980} implies  that for each $i,j=1,\dots,\dd$ and $k,l=1,2$ there is a non-negative and $\tau$-periodic $\tilde{a}^{k,l\pm}_{ij}\in\mathcal{B}(\R^\dd,\R)$ such that 
	\begin{equation*}
		\langle M^i_k(x,\cdot)\pm M^j_l(x,\cdot)\rangle_{t}\,=\,\int_0^t \tilde{a}^{kl\pm}_{ij}\left(\bar X^0(x,s)\right)\, \D s\,, \qquad t\ge0\,.
	\end{equation*}
	Due to boundedness of $b(x),$ $\beta(x)$ and $\upsigma(x)$ we have 
	\begin{align*}
		\uppi^0(\tilde a^{kl\pm}_{ij})&\,=\,\int_{\mathbb{T}_\tau^\dd}\mathbb{E} \left[\int_0^1\tilde a^{kl\pm}_{ij}\left(\bar X^{0,\tau}(x,s)\right)\, \D s\right]\,\uppi^0(\D x)\\
		&\,=\,\int_{\mathbb{T}_\tau^\dd}\mathbb{E} \left[\langle M^i_k(z_x,\cdot)\pm M^j_l(z_x,\cdot)\rangle_{1}\right]\,\uppi^0(\D x)\\
		&\,=\,\int_{\mathbb{T}_\tau^\dd}\mathbb{E} \left[ \left(M^i_k(z_x,1)\pm M^j_l(z_x,1)\right)^2\right]\,\uppi^0(\D x) \\
		&\,\le\,2\int_{\mathbb{T}_\tau^\dd}\mathbb{E} \left[ \left(M^i_k(z_x,1)\right)^2\right]\,\uppi^0(\D x)+2 \int_{\mathbb{T}_\tau^\dd}\mathbb{E} \left[ \left(M^j_l(z_x,1)\right)^2\right]\,\uppi^0(\D x) \,<\,\infty\,,
	\end{align*} 
	where $z_x\in\Pi_\tau^{-1}(\{x\})$ is arbitrary. Set now $\bar a^{kl}_{ij}(x)\df (\tilde a^{kl+}_{ij}(x)-\tilde a^{kl-}_{ij}(x))/4$, and $\bar a^{kl} (x)\df (\bar a^{kl}_{ij}(x))_{i,j=1,\dots,\dd}$. Clearly, for all $k,l=1,2$, $\bar a^{kl}\in\mathcal{B}(\R^\dd,\R^{\dd\times\dd})$ is  $\tau$-periodic, and satisfies $\uppi^0(\HS{\bar a^{kl}})<\infty$ and
	\begin{equation*}
		\langle M_k(x,\cdot),M_l(x,\cdot)\rangle_{t}\,=\,\int_0^t \bar a^{kl}\left(\bar X^0(x,s)\right)\, \D s\,, \qquad t\ge0\,.
	\end{equation*} Furthermore, $\bar a^{11}(x)$ and $\bar a^{22}(x)$ are symmetric and non-negative definite.
	Directly from \Cref{P2.4} and Birkhoff ergodic theorem it follows that for all $t \ge 0$
	\begin{equation}\label{EPT3.1B}
		\varepsilon^2\langle M_k(x/\varepsilon,\cdot),M_l(x/\varepsilon,\cdot) \rangle_{t/\varepsilon^2}\,=\,\varepsilon^2\int_0^{t/\varepsilon^2}\bar a^{kl}\left(\bar X^{0,\tau}(\Pi_\tau(x/\varepsilon),s)\right)\,\D s\,\xrightarrow[\varepsilon\to0]{\Prob\text{-a.s.}}\,\uppi^0(\bar a^{kl})\, t\,.
	\end{equation}

	It remains to determine $\uppi^0(\bar a^{kl})$. For $k=l=1$ dominated convergence theorem implies that 
	\begin{align*}
  		\uppi^0(\bar a^{11}) &\,=\,\lim_{\varepsilon \to 0}\varepsilon^2\int_{\mathbb{T}_\tau^\n}\mathbb{E}\left[\langle M_1(z_x,\cdot),M_1(z_x,\cdot)\rangle_{1/\varepsilon^2}\right]\uppi^0(\D x)\\
  		&\,=\,\lim_{\varepsilon \to 0}\varepsilon^2\int_{\mathbb{T}_\tau^\n}\mathbb{E}\Bigg[\left(\int_0^{\varepsilon^{-2}}\bar{\mathcal{A}}^{0}\beta\bigl(\bar X^{0}(z_x,s)\bigr)\,\D s\right)\left( \int_0^{\varepsilon^{-2}}\bar{\mathcal{A}}^{0}\beta\bigl(\bar X^{0}(z_x,s)\bigr)\,\D s\right)^{\mathrm{T}}\\
  		&\hspace{3cm} -\left(\int_0^{\varepsilon^{-2}}\bar{\mathcal{A}}^{0}\beta\bigl(\bar X^{0}(z_x,s)\bigr)\,\D s\right)\bigl(\beta\bigl(\bar X^0(z_x,1/\varepsilon^2)\bigr)-\beta(z_x)\bigr)^\mathrm{T}\\
  		&\hspace{3cm} -\bigl(\beta\bigl(\bar X^0(z_x,1/\varepsilon^2)\bigr)-\beta(z_x)\bigr)\,\left(\int_0^{\varepsilon^{-2}}\bar{\mathcal{A}}^{0}\beta\bigl(\bar X^{0}(z_x,s)\bigr)\,\D s\right)^{\mathrm{T}}\Bigg]\,\uppi^0(\D x)\\
  		&\,=\,\lim_{\varepsilon \to 0}\varepsilon^2\int_{\mathbb{T}_\tau^\n}\mathbb{E}\Bigg[\left(\int_0^{\varepsilon^{-2}}\bar{\mathcal{A}}^{0,\tau}\beta\bigl(\bar X^{0,\tau}(x,s)\bigr)\,\D s\right)\left( \int_0^{\varepsilon^{-2}}\bar{\mathcal{A}}^{0,\tau}\beta\bigl(\bar X^{0,\tau}(x,s)\bigr)\,\D s\right)^{\mathrm{T}}\\
  		&\hspace{3cm} -\left(\int_0^{\varepsilon^{-2}}\bar{\mathcal{A}}^{0,\tau}\beta\bigl(\bar X^{0,\tau}(x,s)\bigr)\,\D s\right)\bigl(\beta\bigl(\bar X^{0,\tau}(x,1/\varepsilon^2)\bigr)\bigr)^\mathrm{T}\\
  		&\hspace{3cm} -\bigl(\beta\bigl(\bar X^{0,\tau}(x,1/\varepsilon^2)\bigr)\bigr)\,\left(\int_0^{\varepsilon^{-2}}\bar{\mathcal{A}}^{0,\tau}\beta\bigl(\bar X^{0,\tau}(x,s)\bigr)\,\D s\right)^{\mathrm{T}}\Bigg]\,\uppi^0(\D x)\\
  		&\ \ \ \ \ +\int_{\mathbb{T}_\tau^\n}\left(\lim_{\varepsilon \to 0}\varepsilon^2\int_0^{\varepsilon^{-2}}\bar\PP_s^{0,\tau}\bar{\mathcal{A}}^{0,\tau}\beta(x)\,\D s\right)\bigl(\beta(x)\bigr)^\mathrm{T}\uppi^0(\D x)\\
  		&\ \ \ \ \ +\int_{\mathbb{T}_\tau^\n}\beta(x) \left(\lim_{\varepsilon \to 0}\varepsilon^2\int_0^{\varepsilon^{-2}}\bar\PP_s^{0,\tau}\bar{\mathcal{A}}^{0,\tau}\beta(x)\,\D s\right)^\mathrm{T}\uppi^0(\D x)\\
  		&\,=\, \lim_{\varepsilon \to 0}\varepsilon^2\int_{\mathbb{T}_\tau^\n}\mathbb{E}\bigg[2\int_0^{\varepsilon^{-2}}\int_s^{\varepsilon^{-2}}\bar{\mathcal{A}}^{0,\tau}\beta\bigl(\bar X^{0,\tau}(x,v)\bigr)\bigl(\bar{\mathcal{A}}^{0,\tau}\beta\bigl(\bar X^{0,\tau}(x,s)\bigr)\bigr)^{\mathrm{T}}\,\D v\,\D s\\
		&\hspace{3cm} -\left(\int_0^{\varepsilon^{-2}}\bar{\mathcal{A}}^{0,\tau}\beta\bigl(\bar X^{0,\tau}(x,s)\bigr)\,\D s\right)\bigl(\beta\bigl(\bar X^{0,\tau}(x,1/\varepsilon^{2})\bigr)\bigr)^{\mathrm{T}}\\
		&\hspace{3cm} -\bigl(\beta\bigl(\tilde X^{0,\tau}(x,1/\varepsilon^2)\bigr)\bigr)\left(\int_0^{\varepsilon^{-2}}\bar{\mathcal{A}}^{0,\tau}\beta\bigl(\bar X^{0,\tau}(x,s)\bigr)\,\D s\right)^{\mathrm{T}}\bigg]\,\uppi^0(\D x)\\
		&\ \ \ \ \ +\uppi^0(\bar{\mathcal{A}}^{0,\tau}\beta)\, \uppi^0(\beta^\mathrm{T})+ \uppi^0\bigl((\bar{\mathcal{A}}^{0,\tau}\beta)^\mathrm{T}\bigr)\, \uppi^0(\beta)\\
		&\,=\,\lim_{\varepsilon \to 0}\varepsilon^2\int_{\mathbb{T}_\tau^\n}\mathbb{E}\bigg[\int_0^{\varepsilon^{-2}}\left(\int_s^{\varepsilon^{-2}}\bar{\mathcal{A}}^{0,\tau}\beta\bigl(\bar X^{0,\tau}(x,v)\bigr)\,\D v-\beta\bigl(\bar X^{0,\tau}(x,1/\varepsilon^2)\bigr)\right)\\&\hspace{4cm}\bigl(\bar{\mathcal{A}}^{0,\tau}\beta\bigl(\bar X^{0,\tau}(x,s)\bigr)\bigr)^{\mathrm{T}}\,\D s\\
		&\hspace{3cm} +\int_0^{\varepsilon^{-2}}\bar{\mathcal{A}}^{0,\tau}\beta\bigl(\bar X^{0,\tau}(x,s)\bigr)\\&\hspace{4cm}\left(\int_s^{\varepsilon^{-2}}\bar{\mathcal{A}}^{0,\tau}\beta\bigl(\tilde X^{0,\tau}(x,v)\bigr)\,\D v-\beta\bigl(\bar X^{0,\tau}(x,1/\varepsilon^2)\bigr)\right)^{\mathrm{T}}\,\D s\bigg]\,\uppi^0(\D x)\,.
	\end{align*}
	Set $$M_1^\tau(x,t)\,\df\, \beta\bigl(\bar X^{0,\tau}(x,t)\bigr)-\beta(x)-\int_0^{t}\bar{\mathcal{A}}^{0,\tau}\beta\bigl(\bar X^{0,\tau}(x,s)\bigr)\,\D s\,,\qquad t\ge0\,.$$ Clearly, $\process{M_1^\tau}$ is a $\{\mathcal{F}_t\}_{t\ge0}$-martingale. We now have
	\begin{align*}
		\uppi^0(\bar a^{11})&\,=\,\lim_{\varepsilon \to 0}\varepsilon^2\int_{\mathbb{T}_\tau^\n}\mathbb{E}\bigg[\int_0^{\varepsilon^{-2}}\left(M_1^\tau(x,s)-M_1^\tau(x,1/\varepsilon^2)-\beta\bigl(\bar X^{0,\tau}(x,s)\bigr)\right)\\
		&\hspace{4cm}\bigl(\bar{\mathcal{A}}^{0,\tau}\beta\bigl(\bar X^{0,\tau}(x,s)\bigr)\bigr)^{\mathrm{T}}\,\D s\\
		&\ \ \ \  +\int_0^{\varepsilon^{-2}}\bar{\mathcal{A}}^{0,\tau}\beta\bigl(\bar X^{0,\tau}(x,s)\bigr)\left(M^\tau_1(x,s)-M_1^\tau(x,1/\varepsilon^2)-\beta\bigl(\bar X^{0,\tau}(x,s)\bigr)\right)^{\mathrm{T}}\,\D s\bigg]\,\uppi^0(\D x)\\
		&\,=\, -\lim_{\varepsilon \to 0}\varepsilon^2\int_{\mathbb{T}_\tau^\n}\mathbb{E}\bigg[\int_0^{\varepsilon^{-2}}\beta\bigl(\bar X^{0,\tau}(x,s)\bigr)\,\bigl(\bar{\mathcal{A}}^{0,\tau}\beta\bigl(\bar X^{0,\tau}(x,s)\bigr)\bigr)^{\mathrm{T}}\,\D s\\
		&\hspace{3.5cm} +\int_0^{\varepsilon^{-2}}\bar{\mathcal{A}}^{0,\tau}\beta\bigl(\bar X^{0,\tau}(x,s)\bigr)\,\bigl(\beta\bigl(\bar X^{0,\tau}(x,s)\bigr)\bigr)^{\mathrm{T}}\,\D s\bigg]\,\uppi^0(\D x)\\
		&\,=\,-\int_{\mathbb{T}_\tau^\n}\lim_{\varepsilon \to 0}\varepsilon^2\int_0^{\varepsilon^{-2}}\bar\PP^{0,\tau}_s\bigl(\beta(\cdot)\,\bigl(\bar{\mathcal{A}}^{0,\tau}\beta(\cdot)\bigr)^{\mathrm{T}}+\bar{\mathcal{A}}^{0,\tau}\beta(\cdot)\,\bigl(\beta(\cdot)\bigr)^{\mathrm{T}}\bigr)(x)\,\D s\,\uppi^0(\D x)\\
		&\,=\,-\uppi^0\bigl(\beta\,\bigl(\bar{\mathcal{A}}^{0,\tau}\beta\bigl)^{\mathrm{T}}+\bigl(\bar{\mathcal{A}}^{0,\tau}\beta\bigr)\,\beta^{\mathrm{T}}\bigr)\,.
		\end{align*} 
		
		For $k=l=2$ it follows from \cite[Proposition 2.5]{Bhattacharya-1982} that $\bar \uppi^0(\bar a^{22})=\uppi^0(a)$. For mixed terms we have $\langle M_1(x,\cdot),M_2(x,\cdot)\rangle_{t}= \langle M_2(x,\cdot),M_1(x,\cdot)\rangle_{t}^\mathrm{T}$, $t\ge0$. Therefore $\bar a(x) \df \bar a^{12}(x)=\left(\bar a^{21}(x)\right)^\mathrm{T}$, which completes the proof.
\end{proof}

\medskip

Let us now  give several remarks.

\medskip

\begin{remark}\label{R3.2} \begin{itemize}
		\item [\ttup{i}] Notice that 
	$$\mathsf{a}\,=\,\lim_{t \to \infty}\,\frac{1}{t}\int_{\mathbb{T}_\tau^\n}\mathbb{E}\left[\bigl(\bar X^{0}(z_x,t)-z_x-\uppi^0(b)t\bigr)\bigl(\bar X^{0}(z_x,t)-z_x-\uppi^0(b)t\bigr)^{\mathrm{T}}\right]\,\uppi^0(\D x)\,.$$ 
	
	\medskip
	
	\item[\ttup{ii}] In the proof of \Cref{T3.1} we did not use the full strength of the assumptions in \ttup{A1}-\ttup{A3}. It is straightforward to check that the assertion of the theorem holds if for $\tau$-periodic $\upsigma\in\mathcal{B}_b(\R^\n,\R^{\n\times\m})$ and $b\in\mathcal{B}_b(\R^\n,\R^\n)$  the  SDE in \cref{ES2.3} admits a  unique strong solution  which is a time-homogeneous  non-explosive strong Markov process whose projection  on $\mathbb{T}_\tau^\n$ (under $\Pi_\tau(x)$) satisfies the conclusion of \Cref{P2.4} (for $\varepsilon=0$). 
	
		\medskip
		
		\item[\ttup{iii}] In the second part of the proof of \Cref{T3.1} we show that 
		\begin{equation}\label{ER3.2}\left\{\varepsilon\int_0^{\varepsilon^{-2}t}\bar{\mathcal {A}}^0\beta\bigl(\bar X^0_s\bigr)\,\D s\right\}_{t\ge0}\,\xRightarrow[\varepsilon\to0]{({\rm d})}\,\{W^{\mathsf{\Sigma},0}(0,t)\}_{t\ge0}\,,\end{equation} where $\{W^{\mathsf{\Sigma},0}(0,t)\}_{t\ge0}$ is a   $n$-dimensional zero-drift  Brownian motion   (starting from the origin) determined by covariance matrix   $\mathsf{\Sigma}=-\uppi^0(\beta\,(\bar{\mathcal{A}}^{0,\tau}\beta)^{\mathrm{T}}+(\bar{\mathcal{A}}^{0,\tau}\beta)\,\beta^{\mathrm{T}})$.  The proof is based on (a version of) Motoo's theorem given in \cite[Proposition 3.56]{Cinlar-Jacod-Protter-Sharpe-1980}, combined with \Cref{P2.4} and Birkhoff ergodic theorem. 
		It also follows from \cite[Remark  2.1.1]{Bhattacharya-1982} (together with \Cref{P2.4} and \cite[Proposition 2.5]{Bhattacharya-1982}), which is based on a CLT for stationary ergodic sequences given in \cite[Chapter 4.19]{Billingsley-Book-1999}. 	Furthermore, by setting $$h(x)\,\df\, \bar a^{11}(x)+\beta(x)\, \bigl(\bar{\mathcal{A}}^{0}\beta (x)\bigr)^\mathrm{T}+\bar{\mathcal{A}}^{0}\beta (x)\,\bigl(\beta(x)\bigr)^\mathrm{T}\,,\qquad x\in\R^\n\,,$$
		in \cite[Corollary to Theorem 3.1]{Kunita-1969} it has been shown that $$\left\{\beta\bigl(\bar X^0(x,t)\bigr) \bigl(\beta\bigl(\bar X^0(x,t)\bigr)\bigr)^\mathrm{T}-\beta(x)\bigl(\beta(x)\bigr)^\mathrm{T}-\int_0^{t} h\bigl(\bar X^0(x,s)\bigr)\,\D s\right\}_{t\ge0}$$ is a $\{\mathcal{F}_t\}_{t\ge0}$-local martingale. Let  $\{\uptau_k^x\}_{k\ge1}$ be the corresponding localizing sequence. Then, $$\mathbb{E}\left[\beta\bigl(\bar X^0(x,t\wedge\uptau_k^x)\bigr) \bigl(\beta\bigl(\bar X^0(x,t\wedge\uptau_k^x)\bigr)\bigr)^\mathrm{T}\right]-\beta(x)\bigl(\beta(x)\bigr)^\mathrm{T}\,=\,\mathbb{E}\left[\int_0^{t\wedge\uptau_k^x} h\bigl(\bar X^0(x,s)\bigr)\,\D s\right]$$
		for all $k\ge1$ and $t\ge0$.
		Since  $h(x)$ is $\tau$-periodic and satisfies $$\uppi^0\bigl(\HS{h}\bigr)\,<\,\infty\qquad \text{and}\qquad \int_0^t\mathbb{E}\left[\HS{h\bigl(\bar X^0(x,s)\bigr)}\right]\,\D s\,<\,\infty\,,\quad t\ge0\,,$$  by taking $k\to\infty$ in the previous relation, and employing dominated convergence theorem, it follows that $\uppi^0(h)=0$, which again proves that $\uppi^0(\bar a^{11})= -\uppi^0(\beta\,(\bar{\mathcal{A}}^{0,\tau}\beta)^{\mathrm{T}}+(\bar{\mathcal{A}}^{0,\tau}\beta)\,\beta^{\mathrm{T}}).$ See also \cite[VIII.3.65]{Jacod-Shiryaev-2003} for an analogous result.

		Let us also remark that	the CLT of this type   is a very well studied problem in the literature, and it is known that it holds for general ergodic Markov processes (see 
		\cite{Bhattacharya-1982}, \cite[Chapter VIII.3]{Jacod-Shiryaev-2003} and \cite[Chapter 2]{Komorowski-Landim-Olla-Book-2012}).  
		To the best of our knowledge   \cite{Basak-1991} and  \cite{Hashemi-Heunis-2005} are the only two works discussing this problem in the context of    (ergodic) diffusion processes with possibly degenerate diffusion coefficient.
		However, in both works certain ``incremental type'' assumptions on the coefficients have been imposed, which exclude diffusion processes with periodic coefficients.
		Therefore,  it seems that  \Cref{T3.1} is the first result in the literature  showing the relation in \cref{ER3.2} in the case of a periodic diffusion process with degenerate diffusion coefficient.
			\end{itemize}
\end{remark}

\medskip


\subsection{The case  $c(x)\nequiv0$}\label{S3.2} In this case, 
it is not clear that we can perform an  analogous analysis as in \Cref{T3.1}. The difficulty is that the equality in distribution in \cref{ES3.3} does not hold anymore, which implies that the function $\bar a(x)$ (appearing in \Cref{T3.1}) might also depend on the parameter $\varepsilon$. 
In \cite[Lemma 3.2]{Hairer-Pardoux-2008} the authors suggest a solution to this problem but, unfortunately, there seems to be a doubt about its proof. Namely, in the proof  it is assumed that the function $\hat b(x)$ is twice continuously differentiable, but it is  shown that it is  continuously differentiable only. In what follows we  resolve this issue, or at least suggest an alternative approach to the problem. 
We impose an additional assumption on the coefficients $\upsigma(x)$ and $b(x)$, which is taken from \cite{Hairer-Pardoux-2008} (see \cite[Assumption H.3]{Hairer-Pardoux-2008}).
 Let $\upsigma_j(x)\df(\upsigma_{1j}(x),\dots,\upsigma_{\n j}(x))^\mathrm{T}$, $j=1,\dots,\m$, and
let $\mathscr{U}\subseteq[0,\tau]$ be the set where the parabolic H\"{o}rmander condition holds, that is, the set of  $x\in[0,\tau]$ for which the Lie algebra generated by $(b(x),1)\cup\{(\upsigma_1(x),0),\dots,(\upsigma_{\mathrm{m}}(x),0)\}$ spans $\R^{\n+1}$. Observe that $\mathscr{O}\subseteq\mathscr{U}$. 
Assume  the following   

\medskip

\begin{description}
	\item[\ttup{A4}] 
 $\upsigma\in\mathcal{C}^\infty(\R^\n,\R^m)$, $b,c\in\mathcal{C}^\infty(\R^\n,\R^\n)$, and
			$$ \inf_{t>0}\sup_{x\in\R^\n}\mathbb{E}\bigl[\HS{J_x(t)}\,\mathbb{1}_{[t,\infty]}(\bar\uptau^{0,x}_{\mathscr{U}+\tau})\bigr]\,<\,1\,,$$
	\end{description}
where $\{J_x(t)\}_{t\ge0}$ is the Jacobian of the stochastic flow associated to $\process{\bar X^0}$, that is, a solution to \begin{align*}\D J_x(t)&\,=\, \DD b\bigl(\bar X^0(x,t)\bigr)\,J_x(t)\,\D t+\sum_{j=1}^\m\DD\upsigma_j\bigl(\bar X^0(x,t)\bigr)\,J_x(t)\,\D B_j(t)\\
J_x(0)&\,=\,\Id_\n \,.\end{align*}

\noindent As it has been commented in \cite[Remark 2.1]{Hairer-Pardoux-2008}, a simple condition ensuring the above relation to hold is the existence of $t_0>0$ such that $\Prob(\bar\uptau^{0,x}_{\mathscr{U}+\tau}<t_0)=1$ for all $x\in\R^\n$. According to \cite[Lemma II.9.2 and Theorem II.9.5]{Gihman-Skorohod-1972}, smoothness of $\upsigma(x)$, $b(x)$ and $c(x)$ implies that    $\bar\PP_t^{0}f\in\mathcal{C}^k(\R^\n,\R)$  for any $t\ge0$ and $f\in\mathcal{C}_b^k(\R^\n,\R)$, $k=0,1,2$.
Also, under \ttup{A1}-\ttup{A4}, in \cite[Lemma 2.6]{Hairer-Pardoux-2008} it has been shown that there are $\bar\gamma>0$ and $\bar\varGamma>0$, such that \begin{equation}\label{ES3.8}\|\nabla\bar\PP_t^0f(\cdot)\|_\infty\,\le\,\bar\varGamma\bigl(\|f\|_\infty+\|\nabla f(\cdot)\|_\infty \bigr)\E^{-\bar\gamma t}\end{equation} for all $t\ge0$ and $\tau$-periodic $f\in\mathcal{C}^1(\R^\n,\R)$ with $\uppi^0(f)=0$. 
In particular,  $\beta\in \mathcal{C}^1(\R^\n,\R^\n)$.
In what follows we derive an It\^{o}-type formula for the process 
$\{\beta(\bar X^{\varepsilon}(x,t))\}_{t\ge0}$.
Let $(\bar{\mathcal{A}}^{\varepsilon,\tau},\mathcal{D}_{\bar{\mathcal{A}}^{\varepsilon,\tau}})$ be the  $\mathcal{B}_b$-infinitesimal generator   of   $\process{\bar X^{\varepsilon,\tau}}$.
 We start with the following auxiliary lemma.

\medskip

\begin{lemma}\label{L3.3} Assume \ttup{A1}-\ttup{A3},
	and let $f\in\mathcal{C}^2(\R^\n,\R)$ be $\tau$-periodic. Then, $f\in\mathcal{D}_{\bar{\mathcal{A}}^{\varepsilon,\tau}}$, and  $\bar{\mathcal{A}}^{\varepsilon,\tau}f(x)=\bar{\mathcal{A}}^{\varepsilon}f(z_x)$ for all $x\in\mathbb{T}_\tau^\n$ and $z_x\in\Pi^{-1}_\tau(\{x\})$.
\end{lemma}
\begin{proof}
	As we have already commented, $f\in\mathcal{D}_{\bar{\mathcal{A}}^{\varepsilon}}$ and $$\bar{\mathcal{A}}^{\varepsilon}f(x)\,=\,\frac{1}{2}\,\mathrm{Tr}\bigl(a(x)\nabla\nabla^\mathrm{T}f(x)\bigr)+
\bigl( b(x)+\varepsilon c(x)\bigr)^\mathrm{T}\nabla f(x)\,.$$
	Thus,
	\begin{equation*}\lim_{t \to 0}\left\|\left(\bar\PP_t^{\varepsilon,\tau}f(\cdot)-f(\cdot)\right)/t-\bar{\mathcal{A}}^{\varepsilon}f(z_\cdot)\right\|_\infty\,=\,\lim_{t \to 0}\left\|\left(\bar\PP_t^{\varepsilon}f-f\right)/t-\bar{\mathcal{A}}^{\varepsilon}f\right\|_\infty\,=\,0\,.\qedhere\end{equation*}
\end{proof}

\medskip

\noindent 
Let $f\in\mathcal{C}(\R^\n,\R)$ be $\tau$-periodic. 
Define  $$\phi(x)\,\df\,-\int_0^\infty\bar\PP_t^0\bigl(f-\uppi^0(f)\bigr)(x)\,\D t\,,\qquad x\in\R^\n\,.$$ According to  \cref{ES2.7} this function  is well defined, $\tau$-periodic, continuous, and satisfies $\phi\in\mathcal{D}_{\bar{\mathcal{A}}^0}$ and $\bar{\mathcal{A}}^0\phi(x)=f(x)-\uppi^0(f)$.

\medskip

\begin{lemma}\label{L3.4} Assume $f\in\mathcal{C}^2(\R^\n,\R)$. Under \ttup{A1}-\ttup{A4} 
	it holds  that \begin{align*}
	\phi\bigl(\bar X^\varepsilon(x,t)\bigr)\,=\,& \phi(x)+\int_0^t(f-\uppi^0(f))\bigl(\bar X^\varepsilon(x,s)\bigr)\, \D s+\varepsilon \int_0^t\bigl(\bigl(\nabla\phi\bigr)^\mathrm{T}c\bigr)\bigl(\bar X^\varepsilon(x,s)\bigr)\,\D s\\&+\int_0^t\bigl(\bigl(\nabla\phi\bigr)^\mathrm{T}\upsigma\bigr)\bigl(\bar X^\varepsilon(x,s)\bigr)\,\D B^\varepsilon(s)\qquad \forall\,t\ge0\,.
	\end{align*}
\end{lemma}
\begin{proof}
	As we have already commented, $x\mapsto\bar\PP_s^0(f-\uppi^0(f))(x)$ is twice continuously differentiable for any $s\ge0$. It\^{o}'s formula  then gives
	\begin{equation}\label{EPL3.4A}
	\begin{aligned}
	\bar\PP_s^0(f-\uppi^0(f))\bigl(\bar X^\varepsilon(x,t)\bigr)\,=\,& \bar\PP_s^0(f-\uppi^0(f))(x)+\int_0^t\bar{\mathcal{A}}^\varepsilon\bar\PP^0_s(f-\uppi^0(f))\bigl(\bar X^\varepsilon(x,u)\bigr)\, \D u
	\\&+\int_0^t\bigl(\bigl(\nabla \bar\PP^0_s(f-\uppi^0(f))\bigr)^\mathrm{T}\upsigma\bigr)\bigl(\bar X^\varepsilon(x,u)\bigr)\,\D B^\varepsilon(u)\\
	\,=\,& \bar\PP_s^0(f-\uppi^0(f))(x)+\int_0^t\bar{\mathcal{A}}^0\bar\PP^0_s(f-\uppi^0(f))\bigl(\bar X^\varepsilon(x,u)\bigr)\, \D u\\&+\varepsilon \int_0^t\bigl(\bigl(\nabla \bar\PP^0_s(f-\uppi^0(f))\bigr)^\mathrm{T}c\bigr)\bigl(\bar X^\varepsilon(x,u)\bigr)\,\D u
	\\&+\int_0^t\bigl(\bigl(\nabla \bar\PP^0_s(f-\uppi^0(f))\bigr)^\mathrm{T}\upsigma\bigr)\bigl(\bar X^\varepsilon(x,u)\bigr)\,\D B^\varepsilon(u)
	\,.
	\end{aligned}
	\end{equation} 
	By integrating the previous relation with respect to the time variable $s\in[0,\infty)$ (and recalling the definition of the function $\phi(x)$), we arrive at
\begin{equation}\label{EPL3.4B} 
	\begin{aligned}	\phi\bigl(\bar X^\varepsilon(x,t)\bigr)\,=\,& 	\phi(x)-\int_0^\infty\int_0^t\bar{\mathcal{A}}^0\bar\PP^0_s(f-\uppi^0(f))\bigl(\bar X^\varepsilon(x,u)\bigr)\, \D u\,\D s
	\\&+\varepsilon \int_0^t\bigl(\bigl(\nabla\phi\bigr)^\mathrm{T}c\bigr)\bigl(\bar X^\varepsilon(x,u)\bigr)\,\D u
	\\&+\int_0^t\bigl(\bigl(\nabla\phi\bigr)^\mathrm{T}\upsigma\bigr)\bigl(\bar X^\varepsilon(x,u)\bigr)\,\D B^\varepsilon(u)\qquad \forall\,t\ge0\,.\end{aligned}
\end{equation} The last two integrals on the right-hand side  in \cref{EPL3.4B} are well defined, and follow from the last two terms in \cref{EPL3.4A}, because of \cref{ES3.8}. By observing that $\bar{\mathcal{A}}^0\bar\PP^0_sf(x)=\bar\PP^0_s\bar{\mathcal{A}}^0f(x)=\bar\PP^{0,\tau}_s\bar{\mathcal{A}}^{0,\tau}f(\Pi_\tau(x))$ (the last equality follows from \Cref{L3.3}), and $\uppi^0(\bar{\mathcal{A}}^{0,\tau}f)=0$, \cref{ES2.7} implies that the second term on the right-hand side in \cref{EPL3.4B} is well defined. It remains to prove that $$	
-\int_0^\infty\bar{\mathcal{A}}^0\bar\PP^0_s(f-\uppi^0(f))(x)\, \D s\,=\,(f-\uppi^0(f))(x)\qquad\forall\, x\in\R^\n\,.$$
We have 
\begin{align*}
\int_0^\infty\bar{\mathcal{A}}^0\bar\PP^0_t(f-\uppi^0(f))(x)\,\D t\,=\,\int_0^\infty \lim_{s\to0}\frac{\bar\PP^0_{s+t}(f-\uppi^0(f))(x)-\bar\PP^0_t(f-\uppi^0(f))(x)}{s}\,\D t\,.
\end{align*}  By employing It\^{o}'s formula, \Cref{L3.3} and \cref{ES2.7}, we have
\begin{align*}\frac{\|\bar\PP^0_{s+t}(f-\uppi^0(f))-\tilde\PP^0_t(f-\uppi^0(f))\|_\infty}{s}&\,\le\,
\frac{\int_t^{s+t}\|\bar\PP^0_{u}\bar{\mathcal{A}}^0(f-\uppi^0(f))\|_\infty\,\D u}{s}\\&\,\le\,\varGamma\,\|\bar{\mathcal{A}}^0(f-\uppi^0(f))\|_\infty\E^{-\gamma t}
\end{align*}
for all $ s,t\in(0,\infty)$.
The result now follows from the dominated convergence theorem.
\end{proof}

\medskip

We are now ready to prove the main result of this subsection.

	\medskip

		\begin{theorem} \label{T3.5}Under \ttup{A1}-\ttup{A4}, the relation in \cref{ES3.1} holds with $\mathsf{a}$ and $\mathsf{b}$ given in \cref{ES3.2}.
	\end{theorem}
	\begin{proof}
		By combining  \Cref{L3.4} (applied to $b(x)$) with \cref{ES2.2} we have 
		\begin{align*}
			& \varepsilon\bar X^\varepsilon(x/\varepsilon,t/\varepsilon^{2})-\varepsilon^{-1}\uppi^0(b)t- x-\varepsilon\beta\bigl(\bar X^\varepsilon(x/\varepsilon,t/\varepsilon^{2})\bigr)+\varepsilon\beta(x/\varepsilon)\\&\,=\,\varepsilon^2\int_0^{\varepsilon^{-2}t}\bigl((\Id_\n-\DD\beta)\,c\bigr)\bigl(\bar X^\varepsilon(x/\varepsilon,s)\bigr)\,\D s\\
			&\ \ \ \ \ +\varepsilon\int_0^{\varepsilon^{-2}t}\bigl((\Id_\n-\DD\beta)\,\upsigma\bigr)\bigl(\bar X^\varepsilon(x/\varepsilon,s)\bigr)\,\D B^\varepsilon(s)\qquad \forall\, t\ge0\,.
		\end{align*}
		Recall that $X^\varepsilon(x,t)=\varepsilon\bar X^\varepsilon(x/\varepsilon,t/\varepsilon^{2}),$ $t\ge0$. Hence, due to boundedness of $\beta(x)$, $\{ X^{\varepsilon}( x,t)-\varepsilon^{-1} \uppi^0(b)t\}_{t\ge0}$ converges in law if, and only if,
		$\{\varepsilon\bar X^\varepsilon(x/\varepsilon,t/\varepsilon^{2})-\varepsilon^{-1}\uppi^0(b)t-$\linebreak $\varepsilon\beta(\bar X^\varepsilon(x/\varepsilon,t/\varepsilon^{2}))+\varepsilon\beta(x/\varepsilon)\}_{t\ge0}$    converges, and if this is the case the limit is the same.
		Clearly, $\{\varepsilon\bar X^\varepsilon(x/\varepsilon,t/\varepsilon^{2})-\varepsilon^{-1}\uppi^0(b)t-\varepsilon\beta(\bar X^\varepsilon(x/\varepsilon,t/\varepsilon^{2}))+\varepsilon\beta(x/\varepsilon)\}_{t\ge0}$ is a semimartingale  with bounded variation and predictable  quadratic covariation parts  $$\left\{\varepsilon^2\int_0^{\varepsilon^{-2}t}\bigl((\Id_\n-\DD\beta)\,c\bigr)\bigl(\bar X^\varepsilon(x/\varepsilon,s)\bigr)\,\D s\right\}_{t\ge0}\,,$$ and   $$\left\{\varepsilon^2\int_0^{\varepsilon^{-2}t}\bigl((\Id_\n-\DD\beta)\,a\,(\Id_\n-\DD\beta)^{\mathrm{T}}\bigr)\bigl(\bar X^\varepsilon(x/\varepsilon,s)\bigr)\,\D s\right\}_{t\ge0}\,,$$ respectively.  From  \cite[Theorem VI.3.21]{Jacod-Shiryaev-2003} 
		we see that both these processes are tight. Consequently, 
		\cite[Theorem VI.4.18]{Jacod-Shiryaev-2003} implies tightness of  $\{\varepsilon\bar X^\varepsilon(x/\varepsilon,t/\varepsilon^{2})-\varepsilon^{-1}\uppi^0(b)t$\linebreak $-\varepsilon\beta\bigl(\bar X^\varepsilon(x/\varepsilon,t/\varepsilon^{2})\bigr)+\varepsilon\beta(x/\varepsilon)\}_{t\ge0}$.
		To this end, it remains to prove finite-dimensional convergence in law of  $\{\varepsilon\bar X^\varepsilon(x/\varepsilon,t/\varepsilon^{2})-\varepsilon^{-1}\uppi^0(b)t-\varepsilon\beta(\bar X^\varepsilon(x/\varepsilon,t/\varepsilon^{2}))+\varepsilon\beta(x/\varepsilon)\}_{t\ge0}$  to $\{W^{\mathsf{a},0}(x,t)\}_{t\ge0}$.
		According to \cite[Theorem VIII.2.4]{Jacod-Shiryaev-2003} this will hold if $$\varepsilon^2\int_0^{\varepsilon^{-2}t}\bigl((\Id_\n-\DD\beta)\,c\bigr)\bigl(\bar X^\varepsilon(x/\varepsilon,s)\bigr)\,\D s\,\xrightarrow[\varepsilon\to0]{\Prob}\,\mathsf{b}\,t\,,$$ and $$ \varepsilon^2\int_0^{\varepsilon^{-2}t}\bigl((\Id_\n-\DD\beta)\,a\,(\Id_\n-\DD\beta)^{\mathrm{T}}\bigr)\bigl(\bar X^\varepsilon(x/\varepsilon,s)\bigr)\,\D s\,\xrightarrow[\varepsilon\to0]{\Prob}\,\mathsf{a}\,t
		$$ for all $t\ge0$.
		Due to $\tau$-periodicity, we have that
		\begin{align*}&\varepsilon^2\int_0^{\varepsilon^{-2}t}\bigl((\Id_\n-\DD\beta)\,c-\mathsf{b}\bigr)\bigl(\bar X^\varepsilon(x/\varepsilon,s)\bigr)\,\D s\\&\,=\,\varepsilon^2\int_0^{\varepsilon^{-2}t}\bigl((\Id_\n-\DD\beta)\,c-\mathsf{b}\bigr)\bigl(\bar X^{\varepsilon,\tau}(\Pi_\tau(x/\varepsilon),s)\bigr)\,\D s\qquad \forall\,t\ge0\,,\end{align*} and an analogous relation holds for the predictable quadratic covariation part.  We now have \begin{align*}
			&\varepsilon^4\mathbb{E}\Bigg[\left(\int_0^{\varepsilon^{-2}t}\bigl((\Id_\n-\DD\beta)\,c-\uppi^\varepsilon\bigl((\Id_\n-\DD\beta)\,c\bigr)\bigr)\bigl(\bar X^\varepsilon(x/\varepsilon,s)\bigr)\,\D s\right)^{\mathrm{T}}\\&\hspace{0.9cm}\left(\int_0^{\varepsilon^{-2}t}\bigl((\Id_\n-\DD\beta)\,c-\uppi^\varepsilon\bigl((\Id_\n-\DD\beta)\,c\bigr)\bigr)\bigl(\bar X^\varepsilon(x/\varepsilon,s)\bigr)\,\D s\right)\Bigg]\\&\,=\,\varepsilon^4\,\mathbb{E}\Bigg[\left(\int_0^{\varepsilon^{-2}t}\bigl((\Id_\n-\DD\beta)\,c-\uppi^\varepsilon\bigl((\Id_\n-\DD\beta)\,c\bigr)\bigr)\bigl(\bar X^{\varepsilon,\tau}(\Pi_\tau(x/\varepsilon),s)\bigr)\,\D s\right)^{\mathrm{T}}\\&\hspace{1.6cm}\left(\int_0^{\varepsilon^{-2}t}\bigl((\Id_\n-\DD\beta)\,c-\uppi^\varepsilon\bigl((\Id_\n-\DD\beta)\,c\bigr)\bigr)\bigl(\bar X^{\varepsilon,\tau}(\Pi_\tau(x/\varepsilon),s)\bigr)\,\D s\right)\Bigg]\\
			&\,=\,2\varepsilon^4 \int_0^{\varepsilon^{-2}t}\int_0^{s}\mathbb{E}\bigg[\Big(\bigl((\Id_\n-\nabla\beta)c-\uppi^\varepsilon\bigl((\Id_\n-\DD\beta)\,c\bigr)\bigr)\bigl(\bar X^{\varepsilon,\tau}(\Pi_\tau(x/\varepsilon),s)\bigr)\Big)^{\mathrm{T}}\\
			&\hspace{3.6cm}\Big(\bigl((\Id_\n-\DD\beta)\,c-\uppi^\varepsilon\bigl((\Id_\n-\DD\beta)\,c\bigr)\bigr)\bigl(\bar X^{\varepsilon,\tau}(\Pi_\tau(x/\varepsilon),u)\bigr)\Big)\bigg]\D u\,\D s\\
			&\,=\,2\varepsilon^4 \int_0^{\varepsilon^{-2}t}\int_0^{s}\mathbb{E}\bigg[\Big(\bar{\mathcal{P}}^{\varepsilon,\tau}_{s-u}\bigl((\Id_\n-\DD\beta)\,c-\uppi^\varepsilon\bigl((\Id_\n-\DD\beta)\,c\bigr)\bigr)\bigl(\bar X^{\varepsilon,\tau}(\Pi_\tau(x/\varepsilon),u)\bigr)\Big)^{\mathrm{T}}\\&\hspace{3.5cm}\Big(\bigl((\Id_\n-\DD\beta)\,c-\uppi^\varepsilon\bigl((\Id_\n-\DD\beta)\,c\bigr)\bigr)\bigl(\bar X^{\varepsilon,\tau}(\Pi_\tau(x/\varepsilon),u)\bigr)\Big)\bigg]\D u\,\D s\\
			&\,\le\,8\varepsilon^4\varGamma\,\|(\Id_\n-\DD\beta)\,c\|^2_\infty \int_0^{\varepsilon^{-2}t}\int_0^{s}\E^{-\gamma(s-u)}\D u\,\D s\\
			&\,=\,\frac{8\varepsilon^4\varGamma\,\|(\Id_\n-\DD\beta)\,c\|^2_\infty}{\gamma^2}\bigl(\varepsilon^{-2}t+\E^{-\gamma\varepsilon^{-2}t}-1\bigr)\,,
		\end{align*}
		where in the fourth step we employed  \cref{ES2.7}. 
		Thus,
		\begin{align*}&\varepsilon^2\Bigg(\mathbb{E}\Bigg[\left(\int_0^{\varepsilon^{-2}t}\bigl((\Id_\n-\DD\beta)\,c-\mathsf{b}\bigr)\bigl(\bar X^\varepsilon(x/\varepsilon,s)\bigr)\,\D s\right)^{\mathrm{T}}\\&\hspace{1.2cm}\left(\int_0^{\varepsilon^{-2}t}\bigl((\Id_\n-\DD\beta)\,c-\mathsf{b}\bigr)\bigl(\bar X^\varepsilon(x/\varepsilon,s)\bigr)\,\D s\right)\Bigg]\Bigg)^{1/2}\\
			&\,\le\,
			\varepsilon^2\Bigg(\mathbb{E}\Bigg[\left(\int_0^{\varepsilon^{-2}t}\bigl((\Id_\n-\DD\beta)\,c-\uppi^\varepsilon\bigl((\Id_\n-\DD\beta)\,c\bigr)\bigr)\bigl(\bar X^\varepsilon(x/\varepsilon,s)\bigr)\,\D s\right)^{\mathrm{T}}\\&\hspace{1.9cm}\left(\int_0^{\varepsilon^{-2}t}\bigl((\Id_\n-\DD\beta)\,c-\uppi^\varepsilon\bigl((\Id_\n-\DD\beta)\,c\bigr)\bigr)\bigl(\bar X^\varepsilon(x/\varepsilon,s)\bigr)\,\D s\right)\Bigg]\Bigg)^{1/2}\\
			&\ \ \ \ \ +|\uppi^{\varepsilon}\bigl((\Id_\n-\DD\beta)\,c\bigr)-\mathsf{b}|t\,.\end{align*}
		Analogous estimate holds for for the predictable quadratic covariation part. Finally, by letting $\varepsilon\to0$ the result follows from \Cref{P2.5}.
	\end{proof}

	\medskip
	
	Let us now give several remarks.
	
	\medskip

\begin{remark}\label{R1.3}\begin{itemize}
		\item [\ttup{i}]Observe  that when $b(x)\equiv b\in\R^\n$ then $\beta(x)\equiv0$ and, in this case, the conclusion of \Cref{T3.5} holds under \ttup{A1}-\ttup{A3} (that is, assumption \ttup{A4} is not necessary).
		
		\medskip

		\item[\ttup{ii}] By additionally assuming \ttup{A4} in \Cref{T3.1} we can derive a more explicit form of the covariance matrix $\mathsf{a}$. According to \Cref{L3.4} it holds that $$\beta\bigl(\bar X^0(x,t))\,=\, \beta(x)+\int_0^t(b-\uppi^0(b))\bigl(\bar X^0(x,s)\bigr)\, \D s+\int_0^t\bigl(\DD\beta\,\upsigma\bigr)\bigl(\bar X^0(x,s)\bigr)\,\D B(s)\qquad \forall\, t\ge0\,.$$ Now, from \cref{EPT3.1}, \Cref{P2.4}, Birkhoff ergodic theorem and \cite[Proposition 2.5]{Bhattacharya-1982} (analogously as in \cref{EPT3.1B}) it follows that $\uppi^0(\bar a)=-\uppi^0(\DD\beta\,a)$. Thus, $$\mathsf{a}\,=\,\uppi^0\left(a+\bigl(\DD\beta\bigr)\, a+a\,\bigl(\DD\beta\bigr)^\mathrm{T}-\beta\left(\bar{\mathcal{A}}^{0}\beta\right)^{\mathrm{T}}-\left(\bar{\mathcal{A}}^{0}\beta\right)\beta^{\mathrm{T}}\right)\,.$$
		
		Let us remark that an analogous representation of the covariance matrix $\mathsf{a}$ in the uniformly elliptic case has been derived in \cite{Bhattacharya-1985}. More precisely, 
		under the same assumptions as in \cite[Theorem 3.4.3]{Bensoussan-Lions-Papanicolaou-Book-1978} ($a(x)$ is uniformly elliptic and twice continuously differentiable,   and $b\in\mathcal{C}^2(\R^\n,\R)$) in 
		\cite[Theorem 3]{Bhattacharya-1985} 
		it has been shown that  
		$\uppi^0(\D x)$ admits a $\tau$-periodic continuously differentiable density function $\pi^0(x)$ (with respect to the Lebesgue measure on $\mathbb{T}_\tau^\n$), and  $\mathsf{a}$  has the following representation
		\begin{equation}\label{ER3.6}
		\begin{aligned}&\uppi^0\bigl(a-\beta\left(\bar{\mathcal{A}}^{0}\beta\right)^{\mathrm{T}}-\left(\bar{\mathcal{A}}^{0}\beta\right)\beta^{\mathrm{T}}\bigr)\\&+\left(\int_{\mathbb{T}_\tau^\n}\left(\beta_i(x)\sum_{k=1}^\n\partial_k\bigl(\pi^0(x)a_{kj}(x)\bigr)+\beta_j(x)\sum_{k=1}^\n\partial_k\bigl(\pi^0(x)a_{ki}(x)\bigr)\right)\,\D x\right)_{1\le i,j\le\n}\,.\end{aligned}
		\end{equation}
		Recall that in this situation $\beta\in\mathcal{C}^2(\R^\n,\R^\n)$. Also, a direct computation  shows that \cref{ER3.6} transforms to  \cref{ES3.2}, and \textit{vice versa}. 
	\end{itemize}
	
	\end{remark}

\medskip


\section{Homogenization of linear PDEs} \label{S4}

Denote by $\{\hat\PP_t^\varepsilon\}_{t\ge0}$ and $(\hat{\mathcal{A}}^\varepsilon,\mathcal{D}_{\hat{\mathcal{A}}^\varepsilon})$, and $\{\hat\PP_t^0\}_{t\ge0}$ and $(\hat{\mathcal{A}}^0,\mathcal{D}_{\hat{\mathcal{A}}^0})$ the operator semigroup and $\mathcal{B}_b$-infinitesimal generator of $\{ X^\varepsilon(x,t)-\varepsilon^{-1}\uppi^0(b)t\}_{t\ge0}$ and $\process{W^{\mathsf{a},\mathsf{b}}}$, respectively. Observe that for any $f\in\mathcal{C}_{u,b}^2(\R^\n,\R)$, 
$$\hat{\mathcal{A}}^\varepsilon f(x)\,=\,\mathcal{L}^{\varepsilon}f\left(x\right)-\varepsilon^{-1} \uppi^0(b)^{\mathrm{T}}\,\nabla f(x)\,,$$ and $$\hat{\mathcal{A}}^0f(x)\,=\,2^{-1}\mathrm{Tr}\bigl(\mathsf{a}\,\nabla\nabla^\mathrm{T}f(x)\bigr)+
\mathsf{b}^\mathrm{T}\,\nabla f(x)\,.$$
As a consequence of \cite[Theorem 1.1]{Kuhn-2018}, \cite[Theorem 17.25]{Kallenberg-Book-1997} and \Cref{T3.1,T3.5} we have the following.

\medskip

\begin{proposition}\label{P4.1} Assume \ttup{A1}-\ttup{A4} (or \ttup{A1}-\ttup{A3} if $c(x)\equiv0$ or $b(x)\equiv b\in\R^\n$). Then, for any $t_0\ge0$ and $f\in\mathcal{C}_\infty(\R^\n,\R)\cup\{f\in\mathcal{C}(\R^\n,\R)\colon f(x)\ \text{is\ $\tau$-periodic}\}$ it holds that 
	$$\lim_{\varepsilon \to 0}\sup_{0\le t\le t_0}\|\hat\PP_t^\varepsilon f-\hat\PP_t^0f\|_\infty\,=\,0\,,$$ and for any $f\in\mathcal{C}_c^2(\R^\n,\R)\cup\{f\in\mathcal{C}^2(\R^\n,\R)\colon f(x)\ \text{is\ $\tau$-periodic}\}$, $$\lim_{\varepsilon \to 0}\|\hat{\mathcal{A}}^\varepsilon f-\hat{\mathcal{A}}^0f\|_\infty\,=\,0\,.$$ 
\end{proposition}

\medskip

Let us now turn to the problem of homogenization of the problems in \cref{ES1.2,ES1.3}. 
In the sequel, we assume that $\uppi^0(b)=0$. This, in particular, implies that $\{\hat\PP_t^\varepsilon\}_{t\ge0}=\{\PP_t^\varepsilon\}_{t\ge0}$ and $\hat{\mathcal{A}}^\varepsilon =\mathcal{L}^{\varepsilon}$ for $\varepsilon>0$. In the case when 
$b(x)\equiv b\neq0$ (thus $\uppi^0(b)\neq0$)  one can easily construct examples satisfying \ttup{A3}. Recall that in this case \ttup{A4} is not required. On the other hand, when $b(x)$ vanishes \ttup{A3} in general does not have to hold.
A typical example  satisfying $\uppi^0(b)=0$ and \ttup{A1}-\ttup{A4} can be constructed  as follows. For $i=1,\dots,\n$ put $$b_i(x)\,\df\, 2^{-1}\sum_{j=1}^\n\partial_j a_{ij}(x)+\bar b_i(x)\,,\qquad x\in\R^\n\,,$$ where $\bar b_i(x)$ is $\tau$-periodic, of class $\mathcal{C}^\infty$, does not depend on $x_i$, and satisfies $\int_{[0,\tau]}\bar b_i(x)\,\D x=0$. It is then easy to see  that $\uppi^0(\D x)$ is the Lebegues measure on $\mathbb{T}_\tau^\n$ and $\uppi^0(b)=0$, and it is not hard to construct examples satisfying \ttup{A1}-\ttup{A4}.

For instance, let $\n=2$ and $\tau=(10,10)^{\mathrm{T}}$, and take $\tau$-periodic $\upsigma\in C^\infty(\R^2,\R)$ such that $a(x,y)=\upsigma(x,y)\upsigma(x,y)^{\rm T}$ is  positive definite on $\mathscr{B}_{3}(5,5)+\tau$ and $a(x,y) \equiv 0$ on $([0,10]^2\setminus \mathscr{B}_{3}(5,5))+\tau$. Here, $\mathscr{B}_r(x)$ stands for the open ball of radius $r>0$ around $x\in\R^\n.$ For example we can take
\begin{equation*}
	\upsigma(x,y)\,= \,\left(\mathbb{1}_{\mathscr{B}_{3}(5,5)}(x,y)\,\E^{\frac{-1}{9-(x-5)^2-(y-5)^2}}\right)\Id_2\,.
\end{equation*}
It remains to choose $\bar b_1, \bar b_2 \in C^\infty(\R^2,\R)$. Observe that this is enough to satisfy condition \ttup{A1}, condition \ttup{A2} with $\mathscr{O}=\mathscr{B}_{3}(5,5)$, and that in condition \ttup{A4} we have $\mathscr{U}=\mathscr{O}=\mathscr{B}_{3}(5,5)$. 
Notice also that for conditions \ttup{A3} and \ttup{A4} to be satisfied it is enough to take such $b(x,y)$ and $c(x,y)$ that there exists  $t_0>0$ such that $\Prob(\bar{\uptau}^{\varepsilon,(x,y)}_{\mathscr{B}_{3}(5,5)+\tau} < t_0)=1$ for all $(\varepsilon,(x,y))\in [0,\varepsilon_0]\times([0,10]^2\setminus \mathscr{B}_{3}(5,5))$ for some $\varepsilon_0>0$.
Take for instance $\bar b_1(x,y)=\tilde{b}(y)$ and $\bar b_2(x,y)=\tilde{b}(x)$ such that $\tilde{b}(x)$ is $\tau$-periodic and positive for $x \in [0,4]\cup [6,10]$ and on $[4,6]$ define it so that $\int_0^{10}\tilde{b}(x) \,\D x=0$. For example we can take 
\begin{equation*}
	\tilde{b}(x)\,=\, \begin{cases}
		1\,, & x \in [0,4]\cup [6,10]\,, \\
		1-\beta\, \E^{\frac{-1}{1-(x-5)^2}}\,, & x \in (4,6)\,,
	\end{cases} 
\end{equation*}
where $\beta>0$ is such that $\int_0^{10}\tilde{b}(x) \,\D x=0$.
Notice that with such definition of $b(x,y)$ we have that there exists $t_0> 0$ such that for all $(x,y) \in ([0,10]^2\setminus \mathscr{B}_{3}(5,5))+\tau$, $\Prob(\bar{\uptau}^{0,(x,y)}_{\mathscr{B}_{3}(5,5)+\tau} < t_0)=1$. Indeed suppose that we take $(x,y)$ from the central white area in  \Cref{fig:domena}, 
\begin{figure}[h!]
	\centering
	\includegraphics[width=0.42\linewidth]{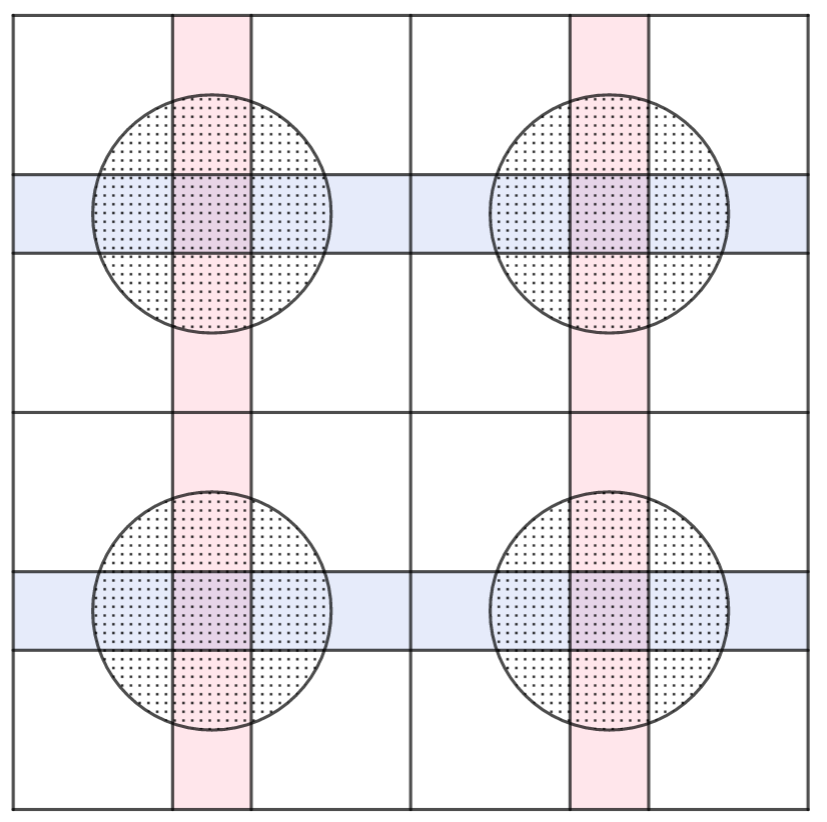}
	\caption{Visualization of different areas of domain for drift term $b(x,y)$.}
	\label{fig:domena}
\end{figure}
while we remain in white area we move at the constant speed diagonally up and to the right (see \Cref{fig:domena2}). 
\begin{figure}[h!]
	\centering
	\includegraphics[width=0.42\linewidth]{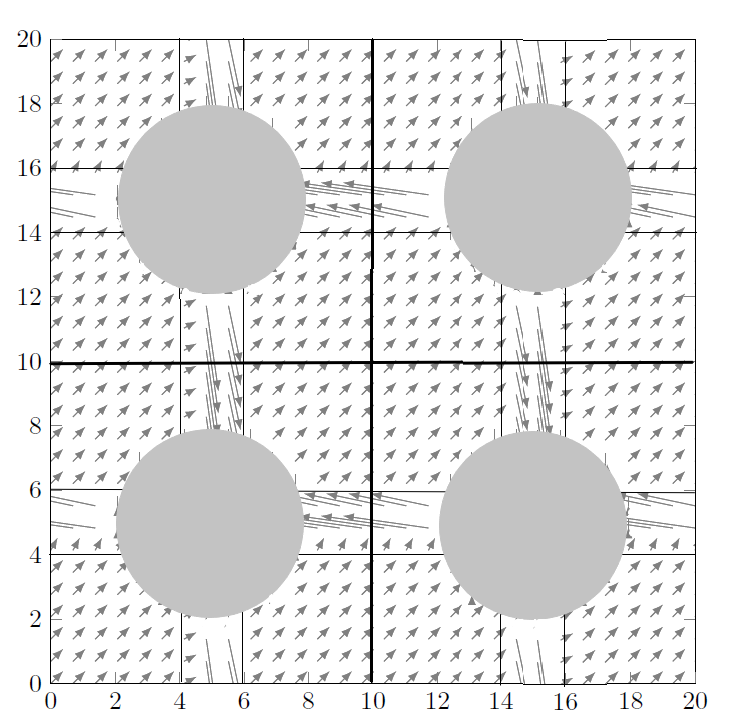}
	\caption{Visualization of the  vector field $b(x,y)$ outside of $\text{supp }\upsigma$.}
	\label{fig:domena2}
\end{figure}
We either hit the upper right circle,  right pink strip or upper blue strip. If we hit the circle, we are done. If we hit the right pink strip (if we hit the upper blue strip we reason in an analogous way),  we continue moving to the right but start to go down. Therefore we either hit the lower right circle of exit the pink strip to the right between two circles. But the later is not possible because $\int_0^{10}\tilde{b}(x) \,\D x=0$ and $\tau$-periodicity of $\tilde{b}(x)$ imply that if the process moves horizontally for $10$ it must vertically return to the same height (see  \Cref{fig:domena2}). Observe that in the previous example $c(x,y)\equiv0.$ However, one can easily see that the same assertion holds with (appropriately chosen) $c(x,y)\not\equiv0$ by choosing $\varepsilon_0$ small enough.


\subsection{The elliptic  problem in \cref{ES1.2}.} 
	Assume that $\mathscr{D}$ is an open bounded subset of $\R^\n$ satisfying the following:
	
	\medskip
	
\begin{itemize}
	\item [(i)] $\mathscr{D}=\left\{x: \mathscr{d}(x)<0 \right\}$ for some $\mathscr{d} \in \mathcal{C}_b^2(\R^\n,\R)$,
	
	\medskip
	
	\item [(ii)] $\left| \nabla \mathscr{d}(x)\right| \ge \delta >0 $ for all $x \in \partial\mathscr{D}$.
\end{itemize}

\medskip

\noindent Further,  suppose that $e\colon \mathscr{D}\to(-\infty,-\alpha]$, $\alpha>0$,  $f\colon \mathscr{D}\to\R$ and $g\colon\partial\mathscr{D}\to\R$ are  continuous, and assume that
$$\{x\in\partial\mathscr{D}\colon \Prob(\hat\uptau^{\varepsilon,x}>0)=0\}$$ is a (topologically) closed set for all $\varepsilon\in[0,\varepsilon_0]$, where $\hat\uptau^{\varepsilon,x}\df \inf\{t\ge0\colon X^\varepsilon(x,t)\notin  \overline{\mathscr{D}}\}$ (recall that $\uppi^0(b)=0$) and $\hat\uptau^{0,x}\df \inf\{t\ge0\colon W^{\mathsf{a},\mathsf{b}}(x,t)\notin  \overline{\mathscr{D}}\}$.
Then, according to \cite[Theorem 3.49]{Pardoux-Rascanu-Book-2014}, \begin{equation}\label{ES4.1}u^\varepsilon(x)\,\df\, \mathbb{E}\left[ g\bigl( X^\varepsilon(x,\hat\uptau^{\varepsilon,x})\bigr)\,\E^{\int_0^{\hat\uptau^{\varepsilon,x}}
	e\left(X^\varepsilon(x,s)/\varepsilon\right)\,\D s}+\int_0^{\hat\uptau^{\varepsilon,x}} f\bigl( X^\varepsilon(x,s)\bigr)\,\E^{\int_0^s
	e\left(X^\varepsilon(x,s)/\varepsilon\right)\,\D u}\,\D s\right]\end{equation}
is a unique continuous viscosity solution (see \cite[Section 6.5]{Pardoux-Rascanu-Book-2014} for the definition of viscosity solutions) to \cref{ES1.2}.

\begin{theorem}\label{T4.5} 
	In addition to the above assumptions, assume \ttup{A1}-\ttup{A4} (or \ttup{A1}-\ttup{A3} if $c(x)\equiv 0$ or $b(x)\equiv 0$), that $$\Prob\left(\left(\nabla\mathscr{d}\bigl(W^{\mathsf{a},\mathsf{b}}(x,\hat\uptau^{0,x})\bigr)\right)^{\mathrm{T}}\mathsf{a}\,\nabla\mathscr{d}\bigl(W^{\mathsf{a},\mathsf{b}}(x,\hat\uptau^{0,x})\bigr)>0\right)\,=\,1\qquad\forall\ x\in\mathscr{D}$$ and that $e(x)$ is $\tau$-periodic (here we consider $e(x)$ as a $\tau$-periodic, continuous and bounded function on $\R^\n$ such that $e(x)\le -\alpha$ for all $x\in \R^\n$). 
	Then, $$ \lim_{\varepsilon \to 0} u^\varepsilon(x)\,=\, u^0(x)\qquad \forall\,x\in \mathscr{D}\,,$$ where $$u^0(x)\,\df\,\mathbb{E}\left[g\bigl( W^{\mathsf{a},\mathsf{b}}(x,\hat\uptau^{0,x})\bigr)\,\E^{\uppi^0(e)\,\hat\uptau^{0,x}}+\int_0^{\hat\uptau^{0,x}}f\bigl( W^{\mathsf{a},\mathsf{b}}(x,s)\bigr)\,\E^{\uppi^0(e)\,s}\D s\right]$$ is a solution to 
	\begin{align*}
	\hat{\mathcal{A}}^0 u^0(x)+\uppi^0(e)\, u^0(x)+f(x)&\,=\, 0\,,\qquad x\in \mathscr{D}\,,\\
	u^0(x)&\,=\,g(x)\,,\qquad x\in \partial \mathscr{D}\,.
	\end{align*}
\end{theorem}
\begin{proof}
	We follow the approach from \cite[Theorem 3.4.5]{Bensoussan-Lions-Papanicolaou-Book-1978}. Define 
	\begin{equation*}
	\zeta^\varepsilon(x,t)\,\df\, \int_{0}^{t} e\bigl(X^\varepsilon(x,s)/\varepsilon\bigr)\,\D s\,=\,\varepsilon^2 \int_0^{\varepsilon^{-2}t}e\bigl(\bar X^\varepsilon(x/\varepsilon,s)\bigr)\,\D s\,.
	\end{equation*}
 Analogously  as in the proof of \Cref{T3.5} we see that
	\begin{equation*}
	\zeta^\varepsilon(x,t) \,=\, \varepsilon^2 \int_0^{\varepsilon^{-2}t}e\left(\bar X^{\varepsilon,\tau} (\Pi_\tau(x/\varepsilon),s)\right)\,\D s \,\xrightarrow[\varepsilon\to0]{{\rm L}^2(\Prob)}\, \uppi^0\left(e\right)t\,,
	\end{equation*} where $\ \xrightarrow[]{{\rm L}^p(\Prob)}$ stands for the convergence in ${\rm L}^p(\Prob)$, $p\ge1$.
	Set $\zeta(x,t)\df \uppi^0\left(e\right) t$. This, together with the fact that the process $\process{\zeta^\varepsilon}$ is tight (due to \cite[Theorem VI.3.21]{Jacod-Shiryaev-2003}), implies that $\process{\zeta^\varepsilon} \xRightarrow[\varepsilon\to0]{({\rm d})} \process{\zeta}$. Since $\process{\zeta}$ is a constant in the space $\mathcal{C}([0,\infty),\R)$, using \Cref{T3.5} and \cite[Theorem 3.9]{Billingsley-Book-1999} we conclude
	\begin{equation}\label{conv.vector.eliptic}
	\process{(X^\varepsilon, \zeta^\varepsilon)}\, \xRightarrow[\varepsilon\to0]{({\rm d})}\, \process{(W^{\mathsf{a},\mathsf{b}}, \zeta)}\,.
	\end{equation}
		We next endow the space $ \mathcal{C}([0,\infty),\R^\n)\times \mathcal{C}([0,\infty),\R)$  with the Borel $\sigma$-algebra generated by the sets $\left\{ (W^{\mathsf{a},\mathsf{b}}, \xi)(x,\cdot) \in  \mathcal{C}([0,\infty),\R^\n)\times \mathcal{C}([0,\infty),\R) \colon (W^{\mathsf{a},\mathsf{b}}, \theta)(x,s) \in B \right\}$ where $s \in [0,\infty)$ and $ B \in \mathfrak{B}(\R^{\n+1})$. The processes $\process{(X^\varepsilon, \xi^\varepsilon)}$ and $\process{(W^{\mathsf{a},\mathsf{b}}, \xi)}$ introduce on\linebreak $\mathcal{C}([0,\infty),\R^\n)\times \mathcal{C}([0,\infty),\R)$  probability measures $\upmu_x^\varepsilon$ and $\upmu_x$, respectively. Observe that\linebreak $\upmu_x^\varepsilon\xRightarrow[\varepsilon\to0]{({\rm w})}\upmu_x$.
		We now define  $\mathrm{F}: \mathcal{C}([0,\infty),\R^\n)\times \mathcal{C}([0,\infty),\R) \to \R\cup\{\infty\}$ with
	\begin{equation*}
	\mathrm{F}(Y,\eta)\,=\,\begin{cases}
	g\bigl(Y\bigl(\uptau (Y)\bigr)\bigr)\,\E^{\eta(\uptau (Y))}+\int_{0}^{\uptau(Y)}f\bigl(Y(t)\bigr)\,\E^{\eta(t)}\,\D t\,, &\uptau(Y)<\infty\ \ \text{and}\ \ \eta(t)\le -\alpha t\ \ \forall t\ge 0\,,
	
	\smallskip
	
	\\
	
	\smallskip
	
	\int_{0}^{\infty}f\bigl(Y(t)\bigr)\,\E^{\eta(t)}\,\D t\,, & \uptau(Y)=\infty\ \ \text{and}\ \ \eta(t)\le -\alpha t\ \ \forall t\ge 0\,,\\

	\infty \,, & \text{otherwise}\,,
	\end{cases}
	\end{equation*}
	where $\uptau(Y)\df\inf \{t\ge 0\colon Y(t) \notin \overline{\mathscr{D}}\}$.  
	Clearly,
	\begin{equation*}
	u^\varepsilon(x)\,=\,\mathbb{E} \left[\mathrm{F}\left(\process{(X^\varepsilon, \zeta^\varepsilon)}\right)\right]\qquad \text{and}\qquad u^0(x)\,=\, \mathbb{E} \left[\mathrm{F}\left(\process{(W^{\mathsf{a},\mathsf{b}}, \zeta)}\right)\right]\,.
	\end{equation*}
The function $\mathrm{F}(Y,\eta)$ has the following properties:
 
 \medskip
 
	\begin{itemize}
		\item[(i)] it is measurable and bounded a.s. with respect to $\upmu_x^\varepsilon$ and $\upmu_x$;
		
		\medskip
		
		\item[(ii)] it is continuous a.s. with respect to $\upmu_x$.
	\end{itemize}

\medskip

\noindent This, together with \cref{conv.vector.eliptic}, gives 
\begin{equation*}
\lim\limits_{\varepsilon \to 0}u^\varepsilon(x)\,=\,u^0(x)\qquad \forall x\in\mathscr{D}\,,
\end{equation*}
which concludes the proof.

To this end, let us  verify (i) and (ii).	To see that  (i) holds, note that if $\eta(t)\le -\alpha t$ for all $t\ge 0$, then
	\begin{align*}
	\left|\mathrm{F}(Y,\eta)\right|&\,\le\, \left\| g \right\|_\infty+\left\| f \right\|_\infty \int_{0}^{\infty}\E^{-\alpha t}\,\D t\,=\,\left\| g \right\|_\infty+ \frac{\left\| f \right\|_\infty}{\alpha}<\infty\,.
	\end{align*}
	Due to the definition of processes $\process{\zeta^\varepsilon}$ and  $\process{\zeta}$, and the fact that for all $x \in \R^\n$ we have $e(x)\le -\alpha$, property (i) follows.
	
	To see property (ii), we need to check that  if $\{Y_n\}_{n\in\N}$ converges  to $Y$ uniformly on compact intervals, then 
\begin{equation}\label{eq}
		\lim_{n\to\infty}\mathrm{F}(Y_n,\eta)\,=\,\mathrm{F}(Y,\eta)
\end{equation} for $\eta(t)=\zeta(x,t)$. Recall that $\zeta(x,t)\le-\alpha t$ for all $x\in\R^\n$ and $t\ge0$.
	The relation in \cref{eq} will follow from the proof of \cite[Lemma 3.4.3]{Bensoussan-Lions-Papanicolaou-Book-1978} where it has been shown that if $\{Y_n\}_{n\in\N}$ converges  to $Y$ uniformly on compact intervals and $\Prob((\nabla\mathscr{d}(W^{\mathsf{a},\mathsf{b}}(x,\hat\uptau^{0,x})))^{\mathrm{T}}\mathsf{a}\,\nabla\mathscr{d}(W^{\mathsf{a},\mathsf{b}}(x,\hat\uptau^{0,x}))>0)=1$ for all $x\in\mathscr{D}$ (see the definition of the function $\beta(t)$ in  \cite[pp. 412]{Bensoussan-Lions-Papanicolaou-Book-1978}), then $\lim_{n \to \infty}\uptau(Y_n)=\uptau(Y)$. Namely, from this fact, and the dominated convergence theorem, we immediately see that  $$\lim_{n \to \infty}\int_{0}^{\uptau(Y_n)}f\bigl(Y_n(t)\bigr)\,\E^{\eta(t)}\D t \,=\, \int_{0}^{\uptau(Y)}f\bigl(Y(t)\bigr)\,\E^{\eta(t)}\D t\,.$$
	To this end, we need to prove that
\begin{equation*}\lim_{n\to\infty}g\bigl(Y_n\bigl(\uptau (Y_n)\bigr)\bigr)\,\E^{\eta(\uptau (Y_n))} \mathbb{1}_{\{\uptau(Y_n)<\infty\}}\,=\,\begin{cases}
	g\bigl(Y\bigl(\uptau (Y)\bigr)\bigr)\,\E^{\eta(\uptau (Y))}\,, &\uptau(Y)<\infty\,,
	
	\smallskip
	
	\\
	
	\smallskip
	
	0\,, &\uptau(Y)=\infty\,.\\

	\end{cases}
	\end{equation*}
 If $\uptau(Y)=\infty$, then $\lim_{n \to \infty}\uptau(Y_n)=\infty$ and since $g(x)$ is bounded and $\eta(t)\le -\alpha t$ for all $t\ge0$, the assertion follows. If $\uptau(Y)<\infty$, then there exist $T>0$ ad $n_Y\in\N$, such that $\uptau(Y_n) \in [0,T]$ for all $n\ge n_Y$. Therefore, 
	\begin{align*}
	&\left|g\bigl(Y_n\bigl(\uptau (Y_n)\bigr)\bigr)\E^{\eta(\uptau (Y_n))}-g\bigl(Y\bigl(\uptau (Y)\bigr)\bigl)\E^{\eta(\uptau (Y))}\right|\\&\,\le\,\|g\|_\infty\left|\E^{\eta(\uptau (Y_n))}-\E^{\eta(\uptau (Y))}\right| +\E^{\eta(\uptau (Y))}\left|g\bigl(Y_n\bigl(\uptau (Y_n)\bigr)\bigr)-g\bigl(Y\bigl(\uptau (Y_n)\bigr)\bigl)\right|\\&\ \ \ \ \ +\E^{\eta(\uptau (Y))}\left|g\bigl(Y\bigl(\uptau (Y_n)\bigr)\bigr)-g\bigl(Y\bigl(\uptau (Y)\bigr)\bigl)\right|\\
	&\,\le\, \|g\|_\infty\left|\E^{\eta(\uptau (Y_n))}-\E^{\eta(\uptau (Y))}\right| + \E^{\eta(\uptau (Y))}\sup_{0\le t \le T} \left|g\bigl(Y_n(t)\bigr)-g\bigl(Y(t)\bigr)\right|\\&\ \ \ \ \ +\E^{\eta(\uptau (Y))}\left|g\bigl(Y\bigl(\uptau (Y_n)\bigr)\bigr)-g\bigl(Y\bigl(\uptau (Y)\bigr)\bigl)\right|\,.
	\end{align*} Clearly, the first and last terms in the above inequality tend to zero as $n$ tends to infinity.
	Suppose that $\limsup_{n \to \infty}\sup_{0\le t \le T} |g(Y_n(t))-g(Y(t))|>0$. Then there exist $\epsilon>0$ and sequences $\{n_k\}_{k \in \N} \subseteq \N$ and $\{t_k\}_{k\in \N} \subseteq [0,T]$, such that $\lim_{k\to\infty}t_k = t \in [0,T]$ and $|g(Y_{n_k}(t_k))-g(Y(t_k))|>\epsilon$ for all $k\in\N$. However, this is not possible since $\lim_{k\to\infty}g(Y(t_k)) =g(Y(t))$, and 
	\begin{align*}
	\lim_{k\to\infty}\left|Y_{n_k}(t_k)-Y(t)\right|&\,\le\, \lim_{k\to\infty}\left|Y_{n_k}(t_k)-Y(t_k)\right|+\lim_{k\to\infty}\left|Y(t_k)-Y(t)\right|\\&\,\le\, \lim_{k\to\infty}\sup_{0\le s \le T}\left|Y_{n_k}(s)-Y(s)\right|+\lim_{k\to\infty}\left|Y(t_k)-Y(t)\right|\\
	&\,=\,0\,.
	\end{align*}
		From this we conclude that 
	\begin{equation*}
	\lim_{n \to \infty}\left|g\bigl(Y_n\bigl(\uptau (Y_n)\bigr)\bigr)\E^{\eta(\uptau (Y_n))}-g\bigl(Y\bigl(\uptau (Y)\bigr)\bigl)\E^{\eta(\uptau (Y))}\right|\,=\, 0\,,
	\end{equation*}
which proves the assertion.
\end{proof}


\subsection{The parabolic  problem in \cref{ES1.3}.} 
Let $d,e\in\mathcal{C}_b(\R^\n,\R)$, and let $f,g\in\mathcal{C}(\R^\n,\R)$ be such that \begin{equation}\label{ES4.2}|f(x)|+|g(x)|\,\le\, K\,(1+|x|^\kappa)\end{equation} for some $\kappa, K>0$ and all $x\in\R^\n$. Then, according to \cite[Theorem 3.43]{Pardoux-Rascanu-Book-2014} (see also \cite[Remark 2.5]{Pardoux-1998}), for any $\varepsilon>0$, 
\begin{equation}\begin{aligned}\label{ES4.3}u^\varepsilon(x,t)\,\df\, \mathbb{E}\bigg[ &g\bigl( X^\varepsilon(x,t)\bigr)\,\E^{\int_0^t\left(\varepsilon^{-1}d\left( X^\varepsilon(x,s)/\varepsilon\right)+e\left(X^\varepsilon(x,s)/\varepsilon\right)\right)\,\D s}\\
&+\int_0^t f\bigl( X^\varepsilon(x,s)\bigr)\,\E^{\int_0^s\left(\varepsilon^{-1}d\left( X^\varepsilon(x,u)/\varepsilon\right)+e\left(X^\varepsilon(x,u)/\varepsilon\right)\right)\,\D u}\,\D s\bigg]\end{aligned}\end{equation}
is a viscosity solution  to 
\cref{ES1.2}. Assume further that $d(x)$ is $\tau$-periodic,  continuously differentiable and such that $\uppi^0(d)=0$ (otherwise we can just replace $d(x)$ by $d(x)-\uppi^0(d)$ in \cref{ES1.2,ES4.3}).
Then, \ttup{A1}-\ttup{A4} imply that $$\delta(x)\,\df\,-\int_0^\infty\bar\PP_t^0d(x)\,\D t\,,\qquad x\in\R^\n\,,$$ is well defined, $\tau$-periodic, continuously differentiable, and satisfies $\delta\in\mathcal{D}_{\bar{\mathcal{A}}^0}$ and $\bar{\mathcal{A}}^0\delta(x)=d(x)$.

\medskip

\begin{theorem}\label{T4.4} In addition to the above assumptions, assume \ttup{A1}-\ttup{A4} (or \ttup{A1}-\ttup{A3} if $c(x)\equiv0$ or $b(x)\equiv 0$, and $d(x)\equiv 0$), $d\in\mathcal{C}^2(\R^\n,\R)$ and that $e(x)$ is $\tau$-periodic. Then, $$ \lim_{\varepsilon \to 0} u^\varepsilon(t,x)\,=\, u^0(t,x)\qquad \forall\,(t,x)\in[0,\infty)\times\R^\n\,,$$ where $$u^0(t,x)\,\df\,\mathbb{E}\bigg[g(\bar W(x,t))\,\E^{\uppi^0\left(2^{-1}\left(\nabla\delta\right)^{\mathrm{T}}a\,\nabla\delta+e-\left(\nabla\delta\right)^{\mathrm{T}}c\right)\,t}+\int_0^tf(\bar W(x,s))\,\E^{\uppi^0\left(2^{-1}\left(\nabla\delta\right)^{\mathrm{T}}a\,\nabla\delta+e-\left(\nabla\delta\right)^{\mathrm{T}}c\right)\,s}\D s\bigg]$$ is a solution to 
	\begin{align*}
	\partial_t u^0(x,t)&\,=\,\mathcal{L}^0 u(x,t)+\uppi^0\bigl(2^{-1}\left(\nabla\delta\right)^{\mathrm{T}}a\,\nabla\delta+e-\left(\nabla\delta\right)^{\mathrm{T}}c\bigr) u^0(x,t)+f(x)\\
	u^0(x,0)&\,=\,g(x)\,,\qquad x\in\R^\n\,,
	\end{align*}
	and $\process{\bar W}$ is a $\n$-dimensional Brownian motion with $\mathcal{B}_b$-infinitesimal generator $\mathcal{L}^0$, determined by drift vector $\bar{\mathsf{b}}\df \mathsf{b}-\uppi^0((\Id_\n-\mathrm{D}\beta)\,a\,\nabla\delta)$ and covariance matrix $\mathsf{a}$.
		\end{theorem}
\begin{proof}  
We first show that $$\lim_{\varepsilon \to 0}\mathbb{E}\left[ g\bigl( X^\varepsilon(x,t)\bigr)\,\E^{\int_0^t\left(\varepsilon^{-1}d\left( X^\varepsilon(x,s)/\varepsilon\right)+e\left(X^\varepsilon(x,s)/\varepsilon\right)\right)\,\D s}\right]\,=\,\mathbb{E}\left[g(\bar W(x,t))\right]\,\E^{\uppi^0\left(2^{-1}\left(\nabla\delta\right)^{\mathrm{T}}a\,\nabla\delta+e-\left(\nabla\delta\right)^{\mathrm{T}}c\right)\,t}\,.$$
	From \Cref{L3.4} we have that 
		\begin{align*}
	\delta\bigl(\bar X^\varepsilon(x,t)\bigr)\,=\,& \delta(x)+\int_0^td\bigl(\bar X^\varepsilon(x,s)\bigr)\, \D s
+\varepsilon \int_0^t\bigl(\left(\nabla\delta\right)^{\mathrm{T}}c\bigr)\bigl(\bar X^\varepsilon(x,s)\bigr)\,\D s\\&+\int_0^t\bigl((\nabla\delta)^{\mathrm{T}}\upsigma\bigr)\bigl(\bar X^\varepsilon(x,s)\bigr)\,\D B^\varepsilon(s)\qquad \forall\,t\ge0\,.
	\end{align*}
	From this we have that \begin{align*}
	&\varepsilon^{-1}\int_0^td\left(\varepsilon^{-1}X^\varepsilon(x,s)\right)\,\D s\\
	&\,=\, \varepsilon\int_0^{\varepsilon^{-2}t}d\left(\bar X^\varepsilon(x/\varepsilon,s)\right)\,\D s\\
	&\,=\,	
	\varepsilon\delta\bigl(\bar X^\varepsilon(x/\varepsilon,t/\varepsilon^2)\bigr)-\varepsilon\delta(x/\varepsilon) - \varepsilon^2 \int_0^{\varepsilon^{-2}t}\bigl(\left(\nabla\delta\right)^{\mathrm{T}}c\bigr)\bigl(\bar X^\varepsilon(x/\varepsilon,s)\bigr)\,\D s\\ & \ \ \ \
	-\varepsilon\int_0^{\varepsilon^{-2}t}\bigl((\nabla\delta)^{\mathrm{T}}\upsigma\bigr)\bigl(\bar X^\varepsilon(x/\varepsilon,s)\bigr)\,\D B^\varepsilon(s)\,.
	\end{align*}
		By assumption,
	\begin{align}\label{EPT4.4A}
	\mathbb{E}\left[\left|g\bigl( \varepsilon\bar X^\varepsilon(x/\varepsilon^2,t/\varepsilon^2)\bigr)\right|^2\right]\,\le\,2K\left(1+\varepsilon^{2\kappa}\mathbb{E}\left[\left|\bar X^\varepsilon(x/\varepsilon,t/\varepsilon^2)\bigr)\right|^{2\kappa}\right]\right)\,.
	\end{align}
	Without loss of generality we may assume that $\kappa\in\N$. By combining \cref{ES2.2} and \Cref{L3.4} we have
	\begin{equation}\label{EPT4.4B}
	\begin{aligned}
	\varepsilon\bar X^\varepsilon(x/\varepsilon,t/\varepsilon^2)\,=\,&x+\varepsilon\beta\bigl(\bar X^\varepsilon(x/\varepsilon,t/\varepsilon^2)\bigr)-\varepsilon\beta(x/\varepsilon)\\&+\varepsilon^2\int_0^{\varepsilon^{-2}t}(c-\mathrm{D}\beta\, c)\bigl(\bar X^\varepsilon(x/\varepsilon,s)\bigr)\,\D s\\ &+\varepsilon\int_0^{\varepsilon^{-2}t}(\upsigma-\mathrm{D}\beta\,\upsigma)\bigl(\bar X^\varepsilon(x/\varepsilon,s)\bigr)\,\D B^{\varepsilon}(s)\,,\qquad t\ge0\,.
	\end{aligned}	
	\end{equation}
	Thus,
	\begin{align*}
	|\varepsilon\bar X^\varepsilon(x/\varepsilon,t/\varepsilon^2)|^{2\kappa}\,\le\,&\bar K\bigg(|x|^{2\kappa}+\varepsilon^{2\kappa}+t^{2\kappa}+\varepsilon^{2\kappa}\left|\int_0^{\varepsilon^{-2}t}(\upsigma-\mathrm{D}\beta\,\upsigma)\bigl(\bar X^\varepsilon(x/\varepsilon,s)\bigr)\,\D B^{\varepsilon}(s)\right|^{2\kappa}\bigg)\,,
	\end{align*}	
	for some $\bar K>0$ which does not depend on $\varepsilon$. By employing  It\^{o}'s formula and Doob's inequality, we conclude \begin{equation}\label{EPT4.4C}\varepsilon^{2\kappa}\mathbb{E}\left[|\bar X^\varepsilon(x/\varepsilon,t/\varepsilon^2)|^{2\kappa}\right]\,\le\,\tilde K\left(|x|^{2\kappa}+\varepsilon^{2\kappa}+t^{\kappa}\right)\,,\end{equation}
	for some $\tilde K>0$ which does not depend on $\varepsilon$. Consequently,
	\begin{align*}
	&\left|u^{\varepsilon}(x,t)-\mathbb{E}\left[g\bigl( \varepsilon\bar X^\varepsilon(x/\varepsilon,t/\varepsilon^2)\bigr)\,\E^{- \varepsilon^2 \int_0^{\varepsilon^{-2}t}\left(\left(\nabla\delta\right)^{\mathrm{T}} c-e\right)\left(\bar X^\varepsilon(x/\varepsilon,s)\right)\,\D s-\varepsilon\int_0^{\varepsilon^{-2}t}\left((\nabla\delta)^{\mathrm{T}}\upsigma\right)\left(\bar X^\varepsilon(x/\varepsilon,s)\right)\,\D B^\varepsilon(s)}\right]\right|\\
	&\,\le\,	\mathbb{E}\left[\left|g\bigl( \varepsilon\bar X^\varepsilon(x/\varepsilon,t/\varepsilon^2)\bigr)\right|^2\right]^{1/2}\mathbb{E}\left[\left|\E^{\varepsilon\delta\left(\bar X^\varepsilon(x/\varepsilon,t/\varepsilon^2)\right)-\varepsilon\delta(x/\varepsilon)}-1\right|^4\right]^{1/4}\\
	&\hspace{0.65cm}\mathbb{E}\left[\E^{-8 \varepsilon^2 \int_0^{\varepsilon^{-2}t}\left((\nabla\delta)^{\mathrm{T}}a\,\nabla\delta\right)\left(\bar X^\varepsilon(x/\varepsilon,s)\right)\,\D s-4\varepsilon\int_0^{\varepsilon^{-2}t}\left((\nabla\delta)^{\mathrm{T}}\upsigma\right)\left(\bar X^\varepsilon(x/\varepsilon,s)\right)\,\D B^\varepsilon(s)}\right]^{1/4}\E^{t\|2(\nabla\delta)^{\mathrm{T}}a\,\nabla\delta-(\nabla\delta)^{\mathrm{T}}c+e\|_\infty}\\
	&\,\le\,\mathbb{E}\left[\left|g\bigl( \varepsilon\bar X^\varepsilon(x/\varepsilon,\varepsilon^{-2}t)\bigr)\right|^2\right]^{1/2}\mathbb{E}\left[\left|\E^{\varepsilon\delta\left(\bar X^\varepsilon(x/\varepsilon,\varepsilon^{-2}t)\right)-\varepsilon\delta(x/\varepsilon)}-1\right|^4\right]^{1/4}\E^{t\|2(\nabla\delta)^{\mathrm{T}}a\,\nabla\delta-(\nabla\delta)^{\mathrm{T}}c+e\|_\infty}\,.
	\end{align*}
	Thus,
$u^\varepsilon(x,t)$ converges as $\varepsilon\to0$ if, and only if,
	$$\mathbb{E}\left[g\bigl( \varepsilon\bar X^\varepsilon(x/\varepsilon,t/\varepsilon^2)\bigr)\,\E^{- \varepsilon^2 \int_0^{\varepsilon^{-2}t}\left((\nabla\delta)^{\mathrm{T}}c-e\right)\left(\bar X^\varepsilon(x/\varepsilon,s)\right)\,\D s-\varepsilon\int_0^{\varepsilon^{-2}t}\left((\nabla\delta)^{\mathrm{T}}\upsigma\right)\left(\bar X^\varepsilon(x/\varepsilon,s)\right)\,\D B^\varepsilon(s)}\right]$$ converges, and if this is the case the limit is the same. 
	Next,
	\begin{equation}
	\begin{aligned}\label{EPT4.4D}
	\bigg|&\mathbb{E}\left[g\bigl( \varepsilon\bar X^\varepsilon(x/\varepsilon,t/\varepsilon^2)\bigr)\,\E^{- \varepsilon^2 \int_0^{\varepsilon^{-2}t}\left((\nabla\delta)^{\mathrm{T}} c-e\right)\left(\bar X^\varepsilon(x/\varepsilon,s)\right)\,\D s-\varepsilon\int_0^{\varepsilon^{-2}t}\left((\nabla\delta)^{\mathrm{T}}\upsigma\right)\left(\bar X^\varepsilon(x/\varepsilon,s)\right)\,\D B^\varepsilon(s)}\right]\\
	&-\E^{- \uppi^0\left((\nabla\delta)^{\mathrm{T}}c-e-2^{-1}(\nabla\delta)^{\mathrm{T}} a\,\nabla\delta\right)\, t}\,\mathbb{E}\bigg[g\bigl( \varepsilon\bar X^\varepsilon(x/\varepsilon,t/\varepsilon^2)\bigr)\\
	&\hspace{4.9cm}\E^{-\frac{\varepsilon^2}{2} \int_0^{\varepsilon^{-2}t}\left((\nabla\delta)^{\mathrm{T}} a\,\nabla\delta\right)\left(\bar X^\varepsilon(x/\varepsilon,s)\right)\,\D s-\varepsilon\int_0^{\varepsilon^{-2}t}\left((\nabla\delta)^{\mathrm{T}}\upsigma\right)\left(\bar X^\varepsilon(x/\varepsilon,s)\right)\,\D B^\varepsilon(s)}\bigg]\bigg|\\
	&\,\le\, \mathbb{E}\bigg[\left|g\bigl( \varepsilon\bar X^\varepsilon(x/\varepsilon,t/\varepsilon^2)\bigr)\right|\\&\hspace{1.15cm} \left|\E^{- \varepsilon^2 \int_0^{\varepsilon^{-2}t}\left((\nabla\delta)^{\mathrm{T}}c-e-2^{-1}(\nabla\delta)^{\mathrm{T}} a\,\nabla\delta\right)\left(\bar X^\varepsilon(x/\varepsilon,s)\right)\,\D s}-\E^{- \uppi^0\left((\nabla\delta)^{\mathrm{T}}c-e-2^{-1}(\nabla\delta)^{\mathrm{T}} a\,\nabla\delta\right)\, t}\right|\\
	&\hspace{1.4cm}\E^{-\frac{\varepsilon^2}{2} \int_0^{\varepsilon^{-2}t}\left((\nabla\delta)^{\mathrm{T}} a\,\nabla\delta\right)\left(\bar X^\varepsilon(x/\varepsilon,s)\right)\,\D s-\varepsilon\int_0^{\varepsilon^{-2}t}\left((\nabla\delta)^{\mathrm{T}}\upsigma\right)\left(\bar X^\varepsilon(x/\varepsilon,s)\right)\,\D B^\varepsilon(s)}\bigg]\\
	&\,\le\, \mathbb{E}\left[\left|g\bigl( \varepsilon\bar X^\varepsilon(x/\varepsilon,t/\varepsilon^2)\bigr)\right|^2\right]^{1/2}\\
	&\ \ \ \ \ \ \mathbb{E}\bigg[\left|\E^{- \varepsilon^2 \int_0^{\varepsilon^{-2}t}\left((\nabla\delta)^{\mathrm{T}}c-e-2^{-1}(\nabla\delta)^{\mathrm{T}} a\,\nabla\delta\right)\left(\bar X^\varepsilon(x/\varepsilon^2,s)\right)\,\D s}-\E^{- \uppi^0\left((\nabla\delta)^{\mathrm{T}}c-e-2^{-1}(\nabla\delta)^{\mathrm{T}} a\,\nabla\delta\right)\, t}\right|^2\\
	&\hspace{1.3cm}\E^{-\varepsilon^2 \int_0^{\varepsilon^{-2}t}\left((\nabla\delta)^{\mathrm{T}} a\,\nabla\delta\right)\left(\bar X^\varepsilon(x/\varepsilon,s)\right)\,\D s-2\varepsilon\int_0^{\varepsilon^{-2}t}\left((\nabla\delta)^{\mathrm{T}}\upsigma\right)\left(\bar X^\varepsilon(x/\varepsilon,s)\right)\,\D B^\varepsilon(s)}\bigg]^{1/2}\,.
\end{aligned}
\end{equation}	
 From  \cref{EPT4.4A,EPT4.4C} we see that the first term on the right-hand side in \cref{EPT4.4D} is uniformly bounded for $\varepsilon$ on finite intervals. For the second term we have that 
	\begin{equation}\label{EPT4.4E}
	\begin{aligned}
	&\mathbb{E}\bigg[\left|\E^{- \varepsilon^2 \int_0^{\varepsilon^{-2}t}\left((\nabla\delta)^{\mathrm{T}}c-e-2^{-1}(\nabla\delta)^{\mathrm{T}} a\,\nabla\delta\right)\left(\bar X^\varepsilon(x/\varepsilon,s)\right)\,\D s}-\E^{- \uppi^0\left((\nabla\delta)^{\mathrm{T}}c-e-2^{-1}(\nabla\delta)^{\mathrm{T}} a\,\nabla\delta\right) \,t}\right|^2\\
	&\hspace{0.7cm}\E^{-\varepsilon^2 \int_0^{\varepsilon^{-2}t}\left((\nabla\delta)^{\mathrm{T}} a\,\nabla\delta\right)\left(\bar X^\varepsilon(x/\varepsilon,s)\right)\,\D s-2\varepsilon\int_0^{\varepsilon^{-2}t}\left((\nabla\delta)^{\mathrm{T}}\upsigma\right)\left(\bar X^\varepsilon(x/\varepsilon,s)\right)\,\D B^\varepsilon(s)}\bigg]\\
	&\,\le\, \mathbb{E}\bigg[\left|\E^{- \varepsilon^2 \int_0^{\varepsilon^{-2}t}\left((\nabla\delta)^{\mathrm{T}}c-e-2^{-1}(\nabla\delta)^{\mathrm{T}} a\,\nabla\delta\right)\left(\bar X^\varepsilon(x/\varepsilon,s)\right)\,\D s}-\E^{- \uppi^0\left((\nabla\delta)^{\mathrm{T}}c-e-2^{-1}(\nabla\delta)^{\mathrm{T}} a\,\nabla\delta\right)\, t}\right|^4\bigg]^{1/2}\\
	& \hspace{0.65cm}\mathbb{E}\bigg[\E^{-2\varepsilon^2 \int_0^{\varepsilon^{-2}t}\left((\nabla\delta)^{\mathrm{T}} a\,\nabla\delta)^{\mathrm{T}}\right)\left(\bar X^\varepsilon(x/\varepsilon,s)\right)\,\D s-4\varepsilon\int_0^{\varepsilon^{-2}t}\left((\nabla\delta)^{\mathrm{T}}\upsigma\right)\left(\bar X^\varepsilon(x/\varepsilon,s)\right)\,\D B^\varepsilon(s)}\bigg]^{1/2}\,.
	\end{aligned}
	\end{equation}
	Clearly, 	 
	\begin{align*}
	&\mathbb{E}\bigg[\E^{-2\varepsilon^2 \int_0^{\varepsilon^{-2}t}\left((\nabla\delta)^{\mathrm{T}} a\,\nabla\delta\right)\left(\bar X^\varepsilon(x/\varepsilon,s)\right)\,\D s-4\varepsilon\int_0^{\varepsilon^{-2}t}\left((\nabla\delta)^{\mathrm{T}}\upsigma\right)\left(\bar X^\varepsilon(x/\varepsilon,s)\right)\,\D B^\varepsilon(s)}\bigg]\\
	&\,\le\,\E^{6t\|(\nabla\delta)^{\mathrm{T}} a\,\nabla\delta\|_\infty}\,.
	\end{align*}
Analogously  as in the proof of \Cref{T3.5} we see that
	\begin{align*}
	&\varepsilon^2 \int_0^{\varepsilon^{-2}t}\left((\nabla\delta)^{\mathrm{T}}c-e-2^{-1}(\nabla\delta)^{\mathrm{T}} a\,\nabla\delta\right)\left(\bar X^\varepsilon(x/\varepsilon,s)\right)\,\D s\\
	&\,=\, \varepsilon^2 \int_0^{\varepsilon^{-2}t}\left((\nabla\delta)^{\mathrm{T}}c-e-2^{-1}(\nabla\delta)^{\mathrm{T}} a\,\nabla\delta\right)\left(\bar X^{\varepsilon,\tau}(\Pi_\tau(x/\varepsilon),s)\right)\,\D s\\&\,\xrightarrow[\varepsilon\to0]{{\rm L}^2(\Prob)}\, \uppi^0\left((\nabla\delta)^{\mathrm{T}}c-e-2^{-1}(\nabla\delta)^{\mathrm{T}} a\,\nabla\delta\right)\, t\,.
	\end{align*} 
	Consequently, Skorohod representation theorem and dominated convergence theorem imply that the first term on the right-hand side in \cref{EPT4.4E} converges to zero as $\varepsilon\to0$.
Thus,  $u^\varepsilon(x,t)$ converges as $\varepsilon\to0$ if, and only if,
	\begin{align*}
	&\E^{- \uppi^0\left({\rm L}^2c-e-2^{-1}(\nabla\delta)^{\mathrm{T}} a\,\nabla\delta\right)\, t}\,\mathbb{E}\bigg[g\bigl( \varepsilon\bar X^\varepsilon(x/\varepsilon,t/\varepsilon^2)\bigr)\\
	&\hspace{4.2cm}\E^{-\frac{\varepsilon^2}{2} \int_0^{\varepsilon^{-2}t}\left((\nabla\delta)^{\mathrm{T}} a\,\nabla\delta\right)\left(\bar X^\varepsilon(x/\varepsilon,s)\right)\,\D s-\varepsilon\int_0^{\varepsilon^{-2}t}\left((\nabla\delta)^{\mathrm{T}}\upsigma\right)\left(\bar X^\varepsilon(x/\varepsilon^2,s)\right)\,\D B^\varepsilon(s)}\bigg]
		\end{align*}
	converges, and if this is the case the limit is the same. 
	Martingale convergence theorem now implies that 
	$$\E^{-\frac{\varepsilon^2}{2} \int_0^{\varepsilon^{-2}t}\left((\nabla\delta)^{\mathrm{T}} a\,\nabla\delta\right)\left(\bar X^\varepsilon(x/\varepsilon,s)\right)\,\D s-\varepsilon\int_0^{\varepsilon^{-2}t}\left((\nabla\delta)^{\mathrm{T}}\upsigma\right)\left(\bar X^\varepsilon(x/\varepsilon,s)\right)\,\D B^\varepsilon(s)} \,\xrightarrow[t\to\infty]{\Prob\text{-}{\rm a.s.\ and\ L^1(\Prob)}}\,  Y(x,\varepsilon)\,,$$ where $Y(x,\varepsilon)\in\mathrm{L}^1(\Prob)$ satisfies 
	$$\mathbb{E}[Y(x,\varepsilon)|\mathcal{F}_{\varepsilon^{-2}t}]\,=\,\E^{-\frac{\varepsilon^2}{2} \int_0^{\varepsilon^{-2}t}\left((\nabla\delta)^{\mathrm{T}} a\,\nabla\delta\right)\left(\bar X^\varepsilon(x/\varepsilon,s)\right)\,\D s-\varepsilon\int_0^{\varepsilon^{-2}t}\left((\nabla\delta)^{\mathrm{T}}\upsigma\right)\left(\bar X^\varepsilon(x/\varepsilon,s)\right)\,\D B^\varepsilon(s)}\,,\qquad t\ge0\,.$$
	Define $\Prob^\varepsilon(\D\omega)\df Y(x,\varepsilon)(\omega)\, \Prob(\D \omega)$.
	Clearly, 
	\begin{align*}&\mathbb{E}^\varepsilon\left[g\bigl( \varepsilon\bar X^\varepsilon(x/\varepsilon,t/\varepsilon^2)\bigr)\right]\\
	&\,=\,\mathbb{E}\left[g\bigl( \varepsilon\bar X^\varepsilon(x/\varepsilon,t/\varepsilon^2)\bigr)\,\E^{-\frac{\varepsilon^2}{2} \int_0^{\varepsilon^{-2}t}\left((\nabla\delta)^{\mathrm{T}} a\,\nabla\delta\right)\left(\bar X^\varepsilon(x/\varepsilon,s)\right)\,\D s-\varepsilon\int_0^{\varepsilon^{-2}t}\left((\nabla\delta)^{\mathrm{T}}\upsigma\right)\left(\bar X^\varepsilon(x/\varepsilon,s)\right)\,\D B^\varepsilon(s)}\right]\,,
	\end{align*} 
	and Girsanov theorem implies that $\bar B^\varepsilon(t)\df B^\varepsilon(t)+\varepsilon\int_0^{t}(\upsigma^{\mathrm{T}}\nabla\delta)\left(\bar X^\varepsilon(x/\varepsilon,s)\right)\D s$, $t\ge0$, is a $\Prob^\varepsilon$-Brownian motion.
	From \cref{EPT4.4B} we have 
		\begin{align*}
	\varepsilon\bar X^\varepsilon(x/\varepsilon,t/\varepsilon^2)\,=\,&x+\varepsilon\beta\bigl(\bar X^\varepsilon(x/\varepsilon,t/\varepsilon^2)\bigr)-\varepsilon\beta(x/\varepsilon)\\&+\varepsilon^2\int_0^{\varepsilon^{-2}t}\left((c-\mathrm{D}\beta\, c)-(\Id_\n-\mathrm{D}\beta)\,a\,\nabla\delta\right)\bigl(\bar X^\varepsilon(x/\varepsilon,s)\bigr)\,\D s\\ &+\varepsilon\int_0^{\varepsilon^{-2}t}(\upsigma-\mathrm{D}\beta\,\upsigma)\bigl(\bar X^\varepsilon(x/\varepsilon,s)\bigr)\,\D \bar B^{\varepsilon}(s)\,.
	\end{align*}	
	It is clear that $\{\varepsilon\bar X^\varepsilon(x/\varepsilon,t/\varepsilon^2)\}_{t\ge0}$ converges in law as $\varepsilon\to0$ if, and only if, $\{\varepsilon\bar X^\varepsilon(x/\varepsilon,t/\varepsilon^2)-\varepsilon\beta(\bar X^\varepsilon(x/\varepsilon,t/\varepsilon^2))+\varepsilon\beta(x/\varepsilon)\}_{t\ge0}$, and if this is the case the limit is the same. The bounded variation and predictable quadratic covariation parts of $\{\varepsilon\bar X^\varepsilon(x/\varepsilon,t/\varepsilon^2)-\varepsilon\beta(\bar X^\varepsilon(x/\varepsilon,t/\varepsilon^2))+\varepsilon\beta(x/\varepsilon)\}_{t\ge0}$ are given by 
	$$\left\{\varepsilon^2\int_0^{\varepsilon^{-2}t}\bigl((c-\mathrm{D}\beta\, c)-(\Id_\n-\mathrm{D}\beta)\,a\,\nabla\delta\bigr)\bigl(\bar X^\varepsilon(x/\varepsilon,s)\bigr)\,\D s\right\}_{t\ge0}\,,$$
	and
	$$\left\{\varepsilon^2\int_0^{\varepsilon^{-2}t}\bigl((\Id_\n-\mathrm{D}\beta)\, a\,(\Id_\n-\mathrm{D}\beta)^{\mathrm{T}}\bigr)\bigl(\bar X^\varepsilon(x/\varepsilon,s)\bigr)\,\D s\right\}_{t\ge0}\,,$$ respectively.
	We will now show that finite-dimensional distributions of $\{\varepsilon\tilde X^\varepsilon(x/\varepsilon,t/\varepsilon^2)$\linebreak$-\varepsilon\beta(\tilde X^\varepsilon(x/\varepsilon,t/\varepsilon^2))+\varepsilon\beta(x/\varepsilon)\}_{t\ge0}$ converge in law to finite-dimensional distributions of  \linebreak $\process{\bar W}$.
	According to \cite[Theorem VIII.2.4]{Jacod-Shiryaev-2003} this will hold if $$\varepsilon^2\int_0^{\varepsilon^{-2}t}\left((c-\mathrm{D}\beta\, c)-(\Id_\n-\mathrm{D}\beta)\,a\,\nabla\delta\right)\bigl(\bar X^\varepsilon(x/\varepsilon,s)\bigr)\,\D s\,\xrightarrow[\varepsilon\to0]{\Prob^\varepsilon}\,\bar{\mathsf{b}}\, t\,,$$ and $$ \varepsilon^2\int_0^{\varepsilon^{-2}t}\bigl((\Id_\n-\mathrm{D}\beta)\, a\,(\Id_\n-\mathrm{D}\beta)^{\mathrm{T}}\bigr)\bigl(\bar X^\varepsilon(x/\varepsilon,s)\bigr)\,\D s\,\xrightarrow[\varepsilon\to0]{\Prob^\varepsilon}\,\mathsf{a}\,t
	$$ for all $t\ge0$.
  We now have 
	\begin{align*}&\mathbb{E}^\varepsilon\left[\varepsilon^2\left|\int_0^{\varepsilon^{-2}t} \left(\left((c-\mathrm{D}\beta\, c)-(\Id_\n-\mathrm{D}\beta)\,a\,\nabla\delta\right)-\bar{\mathsf{b}}\right)\bigl(\bar X^\varepsilon(x/\varepsilon,s)\bigr)\,\D s\right|\right]^2\\
	&\,=\, \mathbb{E}\Bigg[\varepsilon^2\left|\int_0^{\varepsilon^{-2}t} \left(\left((c-\mathrm{D}\beta\, c)-(\Id_\n-\mathrm{D}\beta)\,a\,\nabla\delta\right)-\bar{\mathsf{b}}\right)\bigl(\bar X^\varepsilon(x/\varepsilon,s)\bigr)\,\D s\right|\\
	&\hspace{1.7cm}\E^{-\frac{\varepsilon^2}{2} \int_0^{\varepsilon^{-2}t}\left((\nabla\delta)^{\mathrm{T}} a\,\nabla\delta\right)\left(\bar X^\varepsilon(x/\varepsilon,s)\right)\,\D s-\varepsilon\int_0^{\varepsilon^{-2}t}\left((\nabla\delta)^{\mathrm{T}}\upsigma\right)\left(\bar X^\varepsilon(x/\varepsilon,s)\right)\,\D B^\varepsilon(s)}\Bigg]^2\\
	&\le \varepsilon^4\mathbb{E}\Bigg[\left(\int_0^{\varepsilon^{-2}t}\left(\left((c-\mathrm{D}\beta\, c)-(\Id_\n-\mathrm{D}\beta)\,a\,\nabla\delta\right)-\bar{\mathsf{b}}\right)\bigl(\bar X^\varepsilon(x/\varepsilon,s)\bigr)\,\D s\right)^{\mathrm{T}}\\&\hspace{1.2cm}\left(\int_0^{\varepsilon^{-2}t}\left(\left((c-\mathrm{D}\beta\, c)-(\Id_\n-\mathrm{D}\beta)\,a\,\nabla\delta\right)-\bar{\mathsf{b}}\right)\bigl(\bar X^\varepsilon(x/\varepsilon,s)\bigr)\,\D s\right)\Bigg]\,\E^{t\|(\nabla\delta)^{\mathrm{T}} a\,\nabla\delta\|_\infty}\,.\end{align*}
	Now, as in the proof of \Cref{T3.5} follows that 
	$$\varepsilon^2\int_0^{\varepsilon^{-2}t}\left((c-\mathrm{D}\beta\, c)-(\Id_\n-\mathrm{D}\beta)\,a\,\nabla\delta\right)\bigl(\bar X^\varepsilon(x/\varepsilon,s)\bigr)\,\D s\,\xrightarrow[\varepsilon\to0]{\mathrm{L}^2(\Prob)}\,\bar{\mathsf{b}}\, t\,,$$ which implies 
	$$\varepsilon^2\int_0^{\varepsilon^{-2}t}\left((c-\mathrm{D}\beta\, c)-(\Id_\n-\mathrm{D}\beta)\,a\,\nabla\delta\right)\bigl(\bar X^\varepsilon(x/\varepsilon,s)\bigr)\,\D s\,\xrightarrow[\varepsilon\to0]{\mathrm{L}^1(\Prob^\varepsilon)}\,\bar{\mathsf{b}}\, t\,.$$ Analogous result holds for the predictable quadratic covariation part.
	Thus, finite-dimensional distributions of $\{\varepsilon\tilde X^\varepsilon(x/\varepsilon,t/\varepsilon^2)\}_{t\ge0}$ converge in law to finite-dimensional distributions of \linebreak
	$\{\bar W(x,t)\}_{t\ge0}.$
	It remains to prove that 
	$$\lim_{\varepsilon \to 0}\mathbb{E}^\varepsilon\bigl[g\bigl(\varepsilon\bar X^{\varepsilon}(x/\varepsilon,t/\varepsilon^2)\bigr)\bigr]\,=\,\mathbb{E}\bigl[g\bigl(\bar W(x,t)\bigr)\bigr]\qquad \forall\, t\ge0\,.$$ Without loss of generality we may assume that $f(x)$ is non-negative. From Skorohod representation theorem and Fatou's lemma we conclude
	$$\liminf_{\varepsilon \to 0}\mathbb{E}^\varepsilon\bigl[g\bigl(\varepsilon\bar X^{\varepsilon}(x/\varepsilon,t/\varepsilon^2)\bigr)\bigr]\,\ge\,\mathbb{E}\bigl[g\bigl(\bar W(x,t)\bigr)\bigr]\qquad \forall\, t\ge0\,.$$ To prove the reverse inequality we proceed as follows. For any $t\ge0$ we have
	\begin{align*}
	&\limsup_{\varepsilon \to 0}\mathbb{E}^\varepsilon\bigl[g\bigl(\varepsilon\bar  X^{\varepsilon}(x/\varepsilon,t/\varepsilon^2)\bigr)\bigr]\\ &\,\le\,\limsup_{m\to\infty}\limsup_{\varepsilon \to 0}\mathbb{E}^\varepsilon\bigl[(g\wedge m)\bigl(\varepsilon\bar X^{\varepsilon}(x/\varepsilon,t/\varepsilon^2)\bigr)\bigr]\\&\ \ \ \ +\limsup_{m\to\infty}\limsup_{\varepsilon \to 0}\mathbb{E}^\varepsilon\left[g\bigl(\varepsilon\bar X^{\varepsilon}(x/\varepsilon,t/\varepsilon^2)\bigr)\,\mathbb{1}_{\{g\left(\varepsilon\bar  X^{\varepsilon}(x/\varepsilon,t/\varepsilon^2)\right)\ge m\}}\right]\\
	&\,\le\,\limsup_{m\to\infty}\mathbb{E}\bigl[(g\wedge m)\bigl(\bar W(x,t)\bigr)\bigr]\\&\ \ \ \ +\limsup_{m\to\infty}\limsup_{\varepsilon \to 0}\mathbb{E}^\varepsilon\left[g\bigl(\varepsilon\bar X^{\varepsilon}(x/\varepsilon,t/\varepsilon^2)\bigr)\,\mathbb{1}_{\{g\left(\varepsilon\bar  X^{\varepsilon}(x/\varepsilon,t/\varepsilon^2)\right)\ge m\}}\right]\\
	&\,=\,\mathbb{E}\bigl[g\bigl(\bar W(x,t)\bigr)\bigr] +\limsup_{m\to\infty}\limsup_{\varepsilon \to 0}\mathbb{E}^\varepsilon\left[g\bigl(\varepsilon\bar  X^{\varepsilon}(x/\varepsilon,t/\varepsilon^2)\bigr)\,\mathbb{1}_{\{g\left(\varepsilon\bar  X^{\varepsilon}(x/\varepsilon,t/\varepsilon^2)\right)\ge m\}}\right]\,.
	\end{align*}
	Finally, we show that $$\limsup_{m\to\infty}\limsup_{\varepsilon \to 0}\mathbb{E}^\varepsilon\left[g\bigl(\varepsilon\bar X^{\varepsilon}(x/\varepsilon,t/\varepsilon^2)\bigr)\,\mathbb{1}_{\{g\left(\varepsilon\bar X^{\varepsilon}(x/\varepsilon,t/\varepsilon^2)\right)\ge m\}}\right]\,=\,0\qquad\forall\,t\ge0\,.$$
	We have
	\begin{align*}&\mathbb{E}^\varepsilon\left[g\bigl(\varepsilon\bar X^{\varepsilon}(x/\varepsilon,t/\varepsilon^2)\bigr)\,\mathbb{1}_{\{g\left(\varepsilon\bar X^{\varepsilon}(x/\varepsilon,t/\varepsilon^2)\right)\ge m\}}\right]\\
	&\,\le\,\mathbb{E}^\varepsilon\left[\left|g\bigl(\varepsilon\bar X^{\varepsilon}(x/\varepsilon,t/\varepsilon^2)\bigr)\right|^2\right]^{1/2} \left(\Prob^\varepsilon\left(g\left(\varepsilon\bar X^{\varepsilon}(x/\varepsilon,t/\varepsilon^2)\right)\ge m\right)\right)^{1/2}\\
	&\,\le\,\frac{1}{m}\,\mathbb{E}^\varepsilon\left[\left|g\bigl(\varepsilon\bar X^{\varepsilon}(x/\varepsilon,t/\varepsilon^2)\bigr)\right|^2\right]\,.
	\end{align*} Now, as in \cref{EPT4.4A} we get
	$$\mathbb{E}^\varepsilon\left[\left|g\bigl(\varepsilon\bar X^{\varepsilon}(x/\varepsilon,t/\varepsilon^2)\bigr)\right|^2\right]\,\le\,\hat K(1+|x|^{2\kappa}+\varepsilon^{2\kappa}+t^{\kappa})$$  for some $\hat K>0$   which does not depend on $\varepsilon$. The assertion  now follows.

	To this end it remains to show \begin{align*}&\lim_{\varepsilon \to 0}\mathbb{E}\left[ \int_0^tf\bigl( X^\varepsilon(x,s)\bigr)\,\E^{\int_0^s\left(\varepsilon^{-1}d\left( X^\varepsilon(x,u)/\varepsilon\right)+e\left(X^\varepsilon(x,u)/\varepsilon\right)\right)\,\D u}\,\D s\right]\\&\,=\,\mathbb{E}\left[\int_0^tf(\bar W(x,s))\,\E^{\uppi^0\left(2^{-1}(\nabla\delta)^{\mathrm{T}}a\,\nabla\delta+e-(\nabla\delta)^\mathrm{T}c\right)\,s}\,\D s\right]\,.\end{align*} From the first part of the proof we see that 
	\begin{align*}&\lim_{\varepsilon \to 0}\mathbb{E}\left[ f\bigl( X^\varepsilon(x,s)\bigr)\,\E^{\int_0^s\left(\varepsilon^{-1}d\left( X^\varepsilon(x,u)/\varepsilon\right)+e\left(X^\varepsilon(x,u)/\varepsilon\right)\right)\,\D u}\right]\\&\,=\,\mathbb{E}\left[f(\bar W(x,s))\,\E^{\uppi^0\left(2^{-1}(\nabla\delta)^{\mathrm{T}}a\,\nabla\delta+e-(\nabla\delta)^\mathrm{T}c\right)\,s}\right]\qquad \forall\,s\ge0\,,\end{align*} 
and
\begin{align*}&\mathbb{E}\left[ f\bigl( X^\varepsilon(x,s)\bigr)\,\E^{\int_0^s\left(\varepsilon^{-1}d\left( X^\varepsilon(x,u)/\varepsilon\right)+e\left(X^\varepsilon(x,u)/\varepsilon\right)\right)\,\D u}\right]\\
&\,\le\,\mathbb{E}\left[ \left|f\bigl( X^\varepsilon(x,s)\bigr)\right|^2\right]^{1/2}\mathbb{E}\left[\E^{2\int_0^s\left(\varepsilon^{-1}d\left( X^\varepsilon(x,u)/\varepsilon\right)+e\left(X^\varepsilon(x,u)/\varepsilon\right)\right)\,\D u}\right]^{1/2}\\
&\,\le\, \check{K}\, (1+|x|^\kappa+\varepsilon^\kappa+s^{\kappa/2})\\
&\ \ \ \ \ \  \mathbb{E}\left[\E^{2\varepsilon\delta\left(\bar X^\varepsilon(x/\varepsilon,s/\varepsilon^2)\right)-2\varepsilon\delta(x/\varepsilon) - 2\varepsilon^2 \int_0^{\varepsilon^{-2}s}\left((\nabla\delta)^\mathrm{T}c-e\right))\left(\bar X^\varepsilon(x/\varepsilon,u)\right)\,\D u
	-2\varepsilon\int_0^{\varepsilon^{-2}s}\left((\nabla\delta)^{\mathrm{T}}\upsigma\right)\left(\bar X^\varepsilon(x/\varepsilon,u)\right)\,\D B^\varepsilon(u)}\right]^{1/2}\\
&\,\le\, \check{K}\, (1+|x|^\kappa+\varepsilon^\kappa+s^{\kappa/2})\, \E^{2\varepsilon\,\|\delta\|_\infty+\|(\nabla\delta)^\mathrm{T}c-e-(\nabla\delta)^\mathrm{T}a\,\nabla\delta\|_\infty s}\,,
\end{align*} for some $\check{K}>0$ which does not depend on $\varepsilon$.
The result now follows from the dominated convergence theorem.
	\end{proof}

\medskip

\begin{remark}\label{R4.6}  Grant the assumptions in \Cref{T4.5}.
 Under $\Prob^\varepsilon$ (where $\Prob^0\df\Prob$) the process $\process{\bar X^\varepsilon}$ solves 
		\begin{align*}\D\bar X^\varepsilon(x,t)&\,=\, b\bigl(\bar X^\varepsilon(x,t)\bigr)\,\D t+\varepsilon (c-a\nabla\delta)\bigl(\bar X^\varepsilon(x,t)\bigr)\D t+ \upsigma\bigl(\bar X^\varepsilon(x,t)\bigr)\,\D \bar B^\varepsilon(t)\\
		\bar X^\varepsilon(x,0)&\,=\,x\,, \end{align*}
		and satisfies
		\begin{align*} \bar X^\varepsilon(x,t)\,=\,&x+\beta\bigl(\bar X^\varepsilon(x,t)\bigr)-\beta(x)+\varepsilon\int_0^t\bigl((c-\mathrm{D}\beta\,c)-(\Id_\n-\mathrm{D}\beta)\,a\,\nabla\delta\bigr)\bigl(\bar X^\varepsilon(x,s)\bigr)\,\D s\\&+\int_0^t(\upsigma-\mathrm{D}\beta\,\upsigma)\bigl(\bar X^\varepsilon(x,s)\bigr)\,\D \bar B^\varepsilon(s)\qquad\forall\,t\ge0\,.\end{align*}
		We now easily see that \cref{ES2.4} and \Cref{P2.1} hold under $\Prob^\varepsilon$. Furthermore, if there is $\varepsilon_1>0$ such that \begin{equation}\label{ES4.9}\Prob^\varepsilon\bigl(\bar\uptau^{\varepsilon,x}_{\mathscr{O}+\tau}<\infty\bigr)\,>\,0\qquad \forall\, (\varepsilon,x)\in[0,\varepsilon_1]\times\R^\n\,, \end{equation}  then \Cref{P2.3,P2.4,P2.5} also hold under $\Prob^\varepsilon$. 
		Assume  $f(x)=f_1(x)+f_2(x/\varepsilon)$ and $g(x)=g_1(x)+g_2(x/\varepsilon)$, where $f_1,g_1\in\mathcal{C}(\R^\n,\R)$ satisfy \cref{ES4.2} and $f_2,g_2\in\mathcal{C}(\R^\n,\R)$ are $\tau$-periodic.
		Then, \begin{align*}\lim_{\varepsilon \to 0}\mathbb{E}^\varepsilon\bigl[f\bigl(\varepsilon\bar X^\varepsilon(x/\varepsilon,t/\varepsilon^2)\bigr)\bigr]&\,=\,\mathbb{E}\bigl[f_1\big(\bar W(x,t)\bigr)\bigr]+\lim_{\varepsilon \to 0}\mathbb{E}^\varepsilon\bigl[f_2\bigl(\Pi_\tau\bigl(\bar X^\varepsilon(x/\varepsilon,t/\varepsilon^2)\bigr)\bigr)\bigr]\\&\,=\,\mathbb{E}\bigl[f_1\big(\bar W(x,t)\bigr)\bigr]+\uppi^0(f_2)\qquad\forall\, t\ge0\,,\end{align*} and $$
		\lim_{\varepsilon \to 0}\mathbb{E}^\varepsilon\bigl[f\bigl(\varepsilon\bar X^\varepsilon(x/\varepsilon,t/\varepsilon^2)\bigr)\bigr]\,=\,\mathbb{E}\bigl[f_1\big(\bar W(x,t)\bigr)\bigr]+\uppi^0(f_2)\qquad\forall\, t\ge0\,.$$
		Thus, under \ttup{A1}-\ttup{A4}  (or \ttup{A1}-\ttup{A3} if $c(x)\equiv0$ or $b(x)\equiv 0$, and $d(x)\equiv 0$) and \cref{ES4.9}, 
		$$ \lim_{\varepsilon \to 0} u^\varepsilon(t,x)\,=\, u^0(t,x)\qquad \forall\,(t,x)\in[0,\infty)\times\R^\n\,,$$ where \begin{align*}u^0(t,x)\,\df\,\mathbb{E}\bigg[&\left(g_1(\bar W(x,t))+\uppi^0(g_2)\right)\,\E^{\uppi^0\left(2^{-1}(\nabla\delta)^{\mathrm{T}}a\,\nabla\delta+e-(\nabla\delta)^{\mathrm{T}}c\right)\,t}\\&+\int_0^t\left(f_1(\bar W(x,s))+\uppi^0(f_2)\right)\,\E^{\uppi^0\left(2^{-1}(\nabla\delta)^{\mathrm{T}}a\,\nabla\delta+e-(\nabla\delta)^{\mathrm{T}}c\right)\,s}\D s\bigg]\end{align*} is a solution to 
		\begin{align*}
		\partial_t u^0(x,t)&\,=\,\mathcal{L}^0 u(x,t)+\uppi^0\bigl(2^{-1}(\nabla\delta)^{\mathrm{T}}a\,\nabla\delta+e-(\nabla\delta)^{\mathrm{T}}c\bigr) u^0(x,t)+f_1(x)+\uppi^0(f_2)\\
		u^0(x,0)&\,=\,g_1(x)+\uppi^0(g_2)\,,\qquad x\in\R^\n\,.
		\end{align*}

		\end{remark}

\section*{Acknowledgements}
 Financial support through the \textit{Alexander-von-Humboldt Foundation} and \textit{Croatian Science Foundation} under project 8958  (for N. Sandri\'c), and
Croatian Science Foundation under project 8958 (for I. Valenti\'c)
are  gratefully acknowledged.

\bibliographystyle{abbrv}
\bibliography{References}

\end{document}